\documentclass[12pt,english]{amsart}
\usepackage{amsmath,amsthm,amssymb,amsfonts,amscd, array,latexsym}

\usepackage[noadjust]{cite}
\usepackage[T2A]{fontenc}
\usepackage[cp1251]{inputenc}
\usepackage[english]{babel}
\usepackage{diagmac2} 

\usepackage{tikz}
\usetikzlibrary{positioning,calc,arrows,arrows.meta}

\newtheorem{theorem}{Theorem}[section]
\newtheorem{corollary}[theorem]{Corollary}
\newtheorem{lemma}[theorem]{Lemma}
\newtheorem{proposition}[theorem]{Proposition}

\theoremstyle{definition}
\newtheorem{definition}[theorem]{Definition}
\newtheorem{example}[theorem]{Example}
\newtheorem{problem}[theorem]{Problem}

\theoremstyle{remark}
\newtheorem{remark}[theorem]{Remark}
\numberwithin{equation}{section}

\newcommand{\RR}{\mathbb{R}}

\newcommand{\mc}[1]{\mathcal{#1}}
\newcommand{\mb}[1]{\mathbf{#1}}
\newcommand{\Coh}{\operatorname{Coh}}
\newcommand{\Id}{\operatorname{Id}}

\def\<#1>{\langle #1 \rangle}

\newbox\onebox
\newcommand{\coherent}[1]{%
\mathbin{\setbox\onebox=\hbox{$=$}\lower0.7\ht\onebox\hbox{$\stackrel{#1}{=}$}}}

\begin{document}

\title{Semigroups generated by partitions}

\author{O. Dovgoshey}
\address{\textbf{O. Dovgoshey}\\
Function theory department\\
Institute of Applied Mathematics and Mechanics of NASU\\
Dobrovolskogo str. 1, Slovyansk 84100, Ukraine\\
e-mail: oleksiy.dovgoshey@gmail.com
}

\subjclass[2010]{Primary 20M05}
\keywords{Partition of a set, semigroup of binary relations, band of semigroups.}

\begin{abstract}
Let \(X\) be a nonempty set and \(X^{2}\) be the Cartesian square of \(X\). Some semigroups of binary relations generated partitions of \(X^2\) are studied. In particular, the algebraic structure of semigroups generated by the finest partition of \(X^{2}\) and, respectively, by the finest symmetric partition of \(X^{2}\) are described.
\end{abstract}

\maketitle

\section{Introduction}

Let \(X\) be a set. A \emph{binary relation} on \(X\) is a subset of the Cartesian square
\[
X^2 = X \times X = \{\<x, y> \colon x, y \in X\}.
\]
The \emph{composition} of binary relations \(\psi\) and \(\gamma\) on \(X\) is a binary relation \(\psi \circ \gamma \subseteq X \times X\) for which \(\<x, y> \in \psi \circ \gamma\) holds if and only if there is \(z \in X\) such that \(\<x, z> \in \psi\) and \(\<z, y> \in \gamma\). It is well-known that \(\circ\) is an associative operation on the set of binary relations on \(X\).

Recall that a \emph{semigroup} is a pair \((\mc{S}, *)\) consisting of a nonempty set \(\mc{S}\) and an associative operation \(* \colon \mc{S} \times \mc{S} \to \mc{S}\) which is called the \emph{multiplication} on \(\mc{S}\). As usual, we use the symbol \(x*y\) instead of \(*\<x,y>\) to indicate the result of applying \(*\) to \(\<x, y>\). A semigroup \(\mc{S} = (\mc{S}, *)\) is a \emph{monoid} if there is \(e \in \mc{S}\) such that
\[
e*s = s*e = s
\]
for every \(s \in \mc{S}\). In this case we say that \(e\) is the \emph{identity element} of the semigroup \((\mc{S}, *)\). A \emph{zero} of a semigroup \((\mc{S}, *)\) is an element \(\theta \in \mc{S}\) for which
\[
\theta * s = s * \theta = \theta
\]
holds for every \(s \in \mc{S}\). A set \(A \subseteq \mc{S}\) is a \emph{set of generators} of \((\mc{S}, *)\) if, for every \(s \in \mc{S}\), there is a finite sequence \(s_1, \ldots, s_k\) of elements of \(A\) such that
\[
s = s_1 * \ldots * s_k.
\]
A nonempty subset \(B\) of \(\mc{S}\) is a \emph{subsemigroup} of \((\mc{S}, *)\) if \(x*y \in B\) holds for all \(x\), \(y \in B\).

We denote by \(\mc{B}_{X} = (\mc{B}_{X}, \circ)\) the semigroup of all binary relations defined on a set \(X\) such that the composition \(\circ\) of relations is the multiplication on \(\mc{B}_{X}\). It is well-known that every semigroup \((\mc{H}, *)\) is isomorphic to a subsemigroup of \(\mc{B}_{X}\) for a suitable \(X\). The properties of \(\mc{B}_{X}\) have been investigated by many mathematicians~\cite{Namnak2006, McAlister1971, Tamura1964, Chaudhuri1980, Clifford1970, Diasamidze1990, Diasamidze1990a, Diasamidze1990b, Kim1978, Konieczny1995, Montague1969, Plemmons1970, Plemmons1970a, Preston1973, Schein1976, Zaretskii1962, Zaretskii1963}. In particular, the minimal generating sets for \(\mc{B}_{X}\) were considered in~\cite{Devadze1968} and~\cite{Konieczny2011}. The so-called complete semigroups of binary relations are investigated by Yasha Diasamidze, Shota Makharadze et al. (see, for example, \cite{Avaliani2004, Givradze2003, Diasamidze1990a, Diasamidze2002, Diasamidze2002a, Diasamidze2016, DiMa, Albayrak2018, Givradze2018}).

Following~\cite{Mckenzie1997} we say that a set \(B\) of binary relations on a set \(X\) is \emph{transitive} if for every \(\<x, y> \in X\times X\) there is \(R \in B\) such that \(\<x, y> \in R\). A homomorphism \(\Phi \colon \mc{S} \to \mc{B}_X\) of \((\mc{S}, *)\) is called \emph{transitive} if \(\Phi(\mc{S})\) is a transitive set of relations. A \emph{faithful representation} of a semigroup \((\mc{S}, *)\) by binary relations is a monomorphism \(\mc{S} \to \mc{B}_X\).

Solving a longstanding problem formulated in~\cite{Schein1963} Ralph McKenzie and Boris Schein prove that all semigroups have faithful transitive representations \cite{Mckenzie1997}. Like every outstanding result, the McKenzie-Shein theorem raises a series of related questions. According to this theorem, for every semigroup \((\mc{H}, *)\) there are a monomorphism \(\Phi \colon \mc{S} \to \mc{B}_{X}\) and a set \(A\) of generators of \(\mc{S}\) such that \(\Phi(A)\) is a cover of \(X^{2}\). What can be said about the properties of this cover? In particular, under what conditions is \(\Phi(A)\) a partition of \(X^{2}\)?

\begin{definition}\label{d1.1}
A monomorphism \(\Phi \colon \mc{S} \to \mc{B}_{X}\) is \(d\)-\emph{transitive} (\emph{disjoint-transitive}) if there is a set \(A\) of generators of \((\mc{S}, *)\) such that \(\{\Phi(a) \colon a \in A\}\) is a partition of the set \(X^{2}\) and, if \((\mc{S}, *)\) contains a zero element \(\theta\), the equality \(\Phi(\theta) = \varnothing\) holds.
\end{definition}

It is clear that every \(d\)-transitive monomorphism \(\mc{S} \to \mc{B}_{X}\) is a faithful and transitive representation of \((\mc{S}, *)\).

The following problem seems to be interesting and this is the main object of research in the paper.

\begin{problem}\label{pr2.25}
Describe the algebraic structure of semigroups \((\mc{H}, *)\) admitting \(d\)-transitive monomorphisms \(\mc{H} \to \mc{B}_{X}\).
\end{problem}

The paper is organized as follows.

In Section~2 we consider some basic partitions of Cartesian squares of sets and describe properties of these partitions.

The main results of the paper are formulated and proved in Section~3 and Section~\(4\). Theorem~\ref{ch2:th6}, Proposition~\ref{p3.11} and Theorem~\ref{t2.13} give us a ``purely algebraic'' description of some classes of semigroups \(\mc{H}\) admitting \(d\)-transitive monomorphisms \(\mc{H} \to \mc{B}_X\) for suitable \(X\).

Examples of semigroups \(\mc{H}\) which have no \(d\)-transitive monomorphisms \(\mc{H} \to \mc{B}_{X}\) are given in Proposition~\ref{p3.7} and Proposition~\ref{p4.14}.

\section{Partitions of Cartesian square}

Let \(X\) be a nonempty set and \(P = \{X_j \colon j \in J\}\) be a set of nonempty subsets of \(X\). The set \(P\) is a \emph{partition} of \(X\) if we have
\[
\bigcup_{j \in J} X_j = X \quad \text{and} \quad X_{j_1} \cap X_{j_2} = \varnothing
\]
for all distinct \(j_1\), \(j_2 \in J\). In what follows we will say that the sets \(X_j\), \(j \in J\) are the \emph{blocks} of \(P\).

We say that partitions \(P = \{X_j \colon j \in J\}\) and \(Q = \{X_i \colon i \in I\}\) of a set \(X\) are equal if and only if there is a bijective mapping \(f \colon J \to I\) such that \(X_j = X_{f(j)}\) holds for every \(j \in J\).

\begin{example}\label{ex1.3}
Let \(X\) and \(Y\) be a nonempty sets. If a mapping \(\Psi\colon X \to Y\) is surjective, then the set
\begin{equation}\label{e1.2}
P_{\Psi^{-1}} := \{\Psi^{-1}(y) \colon y \in Y\}
\end{equation}
is a partition of \(X\) with blocks \(\Psi^{-1}(y)\), \(y \in Y\). Conversely, if \(P = \{X_j \colon j \in J\}\) is a partition of \(X\), then the mapping \(F \colon X \to J\) defined by
\[
\bigl(F(x) = j\bigr) \Leftrightarrow \bigl(x \in X_j\bigr)
\]
is surjective and the equality \(P = P_{F^{-1}}\) holds.
\end{example}

Let \(X\) be a set. A binary relation \(R \subseteq X \times X\) is an \emph{equivalence relation} on \(X\) if the following conditions hold for all \(x\), \(y\), \(z \in X\):
\begin{enumerate}
\item \(\<x, x> \in R\), the \emph{reflexive} law;
\item \((\<x, y> \in R) \Leftrightarrow (\<y, x> \in R)\), the \emph{symmetric} law;
\item \(((\<x, y> \in R) \text{ and } (\<y, z> \in R)) \Rightarrow (\<x, z> \in R)\), the \emph{transitive} law.
\end{enumerate}

If \(R\) is an equivalence relation on \(X\), then an \emph{equivalence class} is defined as a subset \([a]_R\) of \(X\) having the form
\begin{equation}\label{e1.1}
[a]_R = \{x \in X \colon \<x, a> \in R\}, a \in X.
\end{equation}

There exists the well-known, one-to-one correspondence between the equivalence relations and the partitions (see, for example, \cite[Chapter~II, \S{}~5]{KurMost} or \cite[Proposition~1.4.6]{Howie2003}).

\begin{proposition}\label{p1.6}
Let \(X\) be a nonempty set. If \(P = \{X_j \colon j \in J\}\) is a partition of \(X\) and \(R_P\) is a binary relation on \(X\) such that, for every \(\<x, y> \in X \times X\),
\begin{equation*}
\bigl(\<x, y> \in R_P\bigr) \Leftrightarrow \bigl(\exists j \in J (x \in X_j \text{ and } y \in X_j)\bigr),
\end{equation*}
then \(R_P\) is an equivalence relation on \(X\) with equivalence classes \(X_j\). Conversely, if \(R\) is an equivalence relation on \(X\), then the set \(P_R\) of all distinct equivalence classes \([a]_R\) is a partition of \(X\) with blocks \([a]_R\).
\end{proposition}

\begin{remark}\label{r3}
If \(X = \varnothing\), then we have \(X \times X = \varnothing\), so that \(\varnothing\) can be considered as a unique equivalence relation on \(\varnothing\). It should be noted that there is no partition of \(\varnothing\) because every block of each partition is nonempty by definition.
\end{remark}

In the following, we systematically use the notation \(R_P\) for the equivalence relation corresponding to partition \(P\) and the notation \(P_R\) for the partition corresponding to equivalence relation \(R\). In particular, for every nonempty set \(X\), Proposition~\ref{p1.6} implies the equality
\begin{equation}\label{e1.3}
P = P_{R_P}
\end{equation}
if \(P\) is a given partition of \(X\) and, respectively, the equality
\begin{equation}\label{e1.4}
R = R_{P_R}
\end{equation}
if \(R\) is a given equivalence relation on \(X\).

The trivial examples of equivalence relations on \(X\) are the Cartesian square \(X^{2}\) and the \emph{diagonal} \(\Delta_{X}\) of \(X\),
\[
\Delta_{X} := \{\<x, x> \colon x \in X\}.
\]
If \(X\) is nonempty, then we have
\[
P_{\Delta_{X}} = \{\{x\} \colon x \in X\}.
\]
Moreover, for a partition \(P = \{X_j \colon j \in J\}\) of \(X\), the equality \(P = P_{X^2}\) holds if and only if \(|J| = 1\).

In the following proposition starting from a partition \(P\) of a set \(X\) we define a partition \(P \otimes P\) of the Cartesian square \(X^{2}\).

\begin{proposition}\label{p1}
Let \(X\) be a nonempty set and \(P = \{X_j\colon j \in J\}\) be a partition of \(X\). Write
\begin{equation}\label{p1:e1}
P \otimes P := \{X_{j_1} \times X_{j_2} \colon \<j_1, j_2> \in J^2\},
\end{equation}
where \(X_{j_1} \times X_{j_2}\) is the Cartesian product of \(X_{j_1}\) and \(X_{j_2}\), and \(J^{2}\) is the Cartesian square of \(J\). Then \(P \otimes P\) is a partition of the Cartesian square \(X^{2}\) with blocks \(X_{j_1} \times X_{j_2}\), \(\<j_1, j_2> \in J^2\).
\end{proposition}

The proof is simple and we omit it here.

\begin{example}\label{ex1.5}
Let \(X\) be a nonempty set and let
\begin{equation*}
P_{\Delta_{X}} = \{\{x\} \colon x \in X\}
\end{equation*}
be a partition of \(X\) corresponding to the diagonal \(\Delta_{X}\) on \(X\). Then \(P_{\Delta_{X}} \otimes P_{\Delta_{X}}\) is a partition of \(X^{2}\) corresponding to the diagonal on \(X^{2}\),
\[
P_{\Delta_{X}} \otimes P_{\Delta_{X}} = P_{\Delta_{X^2}} = \{\{\<x, y>\} \colon \<x, y> \in X^2\}.
\]
\end{example}

\begin{proposition}\label{p2}
Let \(X\) be a nonempty set. If \(R\) is an equivalence relation on \(X\) and \(P_R = \{X_j \colon j \in J\}\) is the corresponding partition of \(X\), then the equality
\begin{equation}\label{p2:e1}
R = \bigcup_{j \in J} X_j^2
\end{equation}
holds.
\end{proposition}

\begin{proof}
Let \(R\) be an equivalence relation on \(X\) and \(P_R = \{X_j \colon j \in J\}\). The set \(\bigcup_{j \in J} X_j^2\) is a subset of \(X^{2}\) and, consequently, it is a binary relation on \(X\). By Proposition~\ref{p1.6}, for every \(\<x_{1}, y_{1}> \in R\) there is \(j_1 \in J\) such that \(x_{1} \in X_{j_1}\) and \(y_{1} \in X_{j_1}\), i.e.,
\[
\<x_{1}, y_{1}> \in X_{j_1}^{2} \subseteq \bigcup_{j \in J} X_j^2.
\]
It implies the inclusion
\begin{equation}\label{p2:e2}
R \subseteq \bigcup_{j \in J} X_j^{2}.
\end{equation}

Now let \(\<x_{0}, y_{0}>\) be an arbitrary point of \(\bigcup_{j \in J} X_j^{2}\). Then there is \(j_0 \in J\) such that \(\<x_{0}, y_{0}> \in X_{j_0}^{2}\), which means \(x_{0} \in X_{j_0}\) and \(y_{0} \in X_{j_0}\). Since \(\{X_j \colon j \in J\}\) is the partition corresponding to \(R\), we have \(\<x_{0}, y_{0}> \in R\) by Proposition~\ref{p1.6}. Consequently, the inclusion
\begin{equation*}
\bigcup_{j \in J} X_j^{2} \subseteq R
\end{equation*}
holds. The last inclusion and~\eqref{p2:e2} imply~\eqref{p2:e1}.
\end{proof}

\begin{example}\label{ex2.7}
Let \(X\) be a nonempty set. Then the equality
\[
\Delta_{X} = \{\{x\}\times \{x\} \colon x \in X\}
\]
holds.
\end{example}

For every partition \(P = \{X_j \colon j \in J\}\) of a nonempty set \(X\) we define a partition \(P \otimes P^1\) of \(X^{2}\) as
\begin{equation}\label{e1.6}
P \otimes P^1 := \{R_P\} \cup \{X_{j_1}\times X_{j_2} \colon \<j_1, j_2> \in \nabla_{J}\},
\end{equation}
where \(\Delta_{J}\) is the diagonal of \(J\) and \(\nabla_{J} := J^2 \setminus \Delta_{J}\). We will also consider partitions \(P \otimes P_S\) and \(P \otimes P_S^1\) defined by the rules:
\begin{enumerate}
\item A subset \(B\) of \(X^{2}\) is a block of \(P \otimes P_S\) if and only if there are \(j_1\), \(j_2 \in J\) such that
\begin{equation}\label{e1.7}
B = (X_{j_1} \times X_{j_2}) \cup (X_{j_2} \times X_{j_1});
\end{equation}
\item A subset \(B\) of \(X^{2}\) is a block of \(P \otimes P_S^1\) if and only if either \(B = R_P\) or there are \emph{distinct} \(j_1\), \(j_2 \in J\) such that \eqref{e1.7} holds.
\end{enumerate}

\begin{figure}[htb]
\begin{tikzpicture}
\def\xx{0cm};
\def\yy{0cm};
\def\dx{1cm};
\def\dy{1cm};
\draw (\xx-0.1*\dx, \yy) -- (\xx, \yy) node [below left] {\(0\)} -- (\xx + \dx, \yy) node [below] {\(\frac{1}{3}\)} -- (\xx + 2*\dx, \yy) node [below] {\(\frac{2}{3}\)} -- (\xx+3*\dx, \yy) node [below] {\(1\)} -- (\xx+3.1*\dx, \yy);
\draw (\xx,\yy-0.1*\dy) -- (\xx,\yy+\dy) node [left] {\(\frac{1}{3}\)} -- (\xx,\yy+ 2*\dy) node [left] {\(\frac{2}{3}\)} -- (\xx, \yy+3*\dy) node [left] {\(1\)} -- (\xx, \yy+3.5*\dy);

\filldraw [fill=white, draw=black] (\xx, \yy) -- (\xx+\dx, \yy) -- (\xx+\dx, \yy+\dy) -- (\xx, \yy+\dy) -- cycle;
\filldraw [fill=orange, draw=black] (\xx+\dx, \yy) -- (\xx+2*\dx, \yy) -- (\xx+2*\dx, \yy+\dy) -- (\xx+\dx, \yy+\dy) -- cycle;
\filldraw [fill=green, draw=black] (\xx+2*\dx, \yy) -- (\xx+3*\dx, \yy) -- (\xx+3*\dx, \yy+\dy) -- (\xx+2*\dx, \yy+\dy)-- cycle;

\def\yy{\dy};
\filldraw [fill=orange, draw=black] (\xx, \yy) -- (\xx+\dx, \yy) -- (\xx+\dx, \yy+\dy) -- (\xx, \yy+\dy) -- cycle;
\filldraw [fill=white, draw=black] (\xx+\dx, \yy) -- (\xx+2*\dx, \yy) -- (\xx+2*\dx, \yy+\dy) -- (\xx+\dx, \yy+\dy) -- cycle;
\filldraw [fill=violet, draw=black] (\xx+2*\dx, \yy) -- (\xx+3*\dx, \yy) -- (\xx+3*\dx, \yy+\dy) -- (\xx+2*\dx, \yy+\dy) -- cycle;

\def\yy{2*\dy};
\filldraw [fill=green, draw=black] (\xx, \yy) -- (\xx+\dx, \yy) -- (\xx+\dx, \yy+\dy) -- (\xx, \yy+\dy) -- cycle;
\filldraw [fill=violet, draw=black] (\xx+\dx, \yy) -- (\xx+2*\dx, \yy) -- (\xx+2*\dx, \yy+\dy) -- (\xx+\dx, \yy+\dy) -- cycle;
\filldraw [fill=white, draw=black] (\xx+2*\dx, \yy) -- (\xx+3*\dx, \yy) -- (\xx+3*\dx, \yy+\dy) -- (\xx+2*\dx, \yy+\dy) -- cycle;
\end{tikzpicture}
\caption{The partition \(P \otimes P_S^{1}\) corresponding to trichotomy \(P = \{X_0, X_1, X_2\}\). Here \(R_P\) is white, \(R_1\) is orange, \(R_2\) is green and \(R_3\) is violet.}
\label{ch2:fig1}
\end{figure}
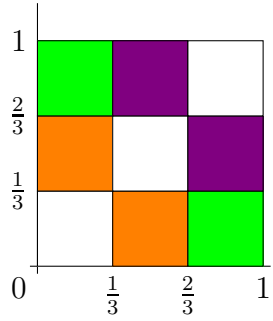

\begin{example}\label{ch2:ex11}
Let \(\RR\) be the field of real numbers and let
\begin{align*}
X & := \{x \in \RR \colon 0 \leqslant x \leqslant 1\}, & X_0 &:= \left\{x \in \RR \colon 0 \leqslant x \leqslant \frac{1}{3}\right\},\\
X_1 &:= \left\{x \in \RR \colon \frac{1}{3} < x < \frac{2}{3}\right\}, & X_2 &:= \left\{x \in \RR \colon \frac{2}{3} \leqslant x \leqslant 1\right\}.
\end{align*}
Then the trichotomy \(P = \{X_0, X_1, X_2\}\) is a partition of \(X\). Write
\begin{align*}
R_1 &:= (X_0 \times X_1) \cup (X_1 \times X_0), & R_2 &:= (X_0 \times X_2) \cup (X_2 \times X_0),\\
R_3 &:= (X_1 \times X_2) \cup (X_2 \times X_1), & R_P &:= \bigcup_{j=0}^2 X_j^2.
\end{align*}
Then \(P \otimes P_S^1 = \{R_P, R_1, R_2, R_3\}\) holds (see Figure~\ref{ch2:fig1}).
\end{example}

The partitions \(P \otimes P\), \(P \otimes P^1\), \(P \otimes P_S\) and \(P \otimes P_S^1\) can be characterized as the smallest elements of corresponding subsets of the partially ordered set of all partitions of \(X^{2}\).

\begin{definition}\label{d1.7}
Let \(X\) be a nonempty set and let \(P_1\) and \(P_2\) be partitions of \(X\). The partition \(P_{1}\) is \emph{finer} than the partition \(P_{2}\) if the inclusion
\[
[x]_{R_{P_1}} \subseteq [x]_{R_{P_2}}
\]
holds for every \(x \in X\), where \(R_{P_1}\) and \(R_{P_2}\) are equivalence relations corresponding to \(P_1\) and \(P_2\) respectively.
\end{definition}

If \(P_1\) is finer than \(P_2\), then we write \(P_{1} \leqslant_{X} P_{2}\) and say that \(P_{1}\) is a \emph{refinement} of \(P_{2}\).

\begin{remark}\label{r2.10}
Using Proposition~\ref{p2}, we see that if \(P_1\) and \(P_2\) are partitions of \(X\), then \(P_1\) is a refinement of \(P_2\) if and only if the inclusion \(R_{P_1} \subseteq R_{P_2}\) holds.
\end{remark}

\begin{example}\label{ex1.8}
Let \(X\) be a nonempty set and let \(P\) be a partition of \(X\). Then
\begin{gather*}
P \otimes P \leqslant_{X^{2}} P \otimes P^{1} \leqslant_{X^{2}} P \otimes P_{S}^{1}\\
\intertext{and}
P \otimes P \leqslant_{X^{2}} P \otimes P_{S} \leqslant_{X^{2}} P \otimes P_{S}^{1}
\end{gather*}
hold.
\end{example}

Recall that a reflexive and transitive binary relation \(\preccurlyeq\) on a set \(Y\) is a \emph{partial order} on \(Y\) if the following \emph{antisymmetric law}
\[
\bigl(\<x, y> \in \preccurlyeq \text{ and } \<y, x> \in \preccurlyeq \bigr) \Rightarrow (x = y)
\]
holds for all \(x\), \(y \in Y\).

In what follows, for partial order \(\preccurlyeq\), we write \(x \preccurlyeq y\) instead of \(\<x, y> \in \preccurlyeq\).

The following proposition is well-known (see, for example, {\cite[Example~4, \S{} 9 Chapter~2]{KurMost}}) and directly follows from Remark~\ref{r2.10}.

\begin{proposition}\label{p1.9}
Let \(X\) be a nonempty set and let \(\mathbf{\Pi}(X)\) be the set of all partitions of \(X\). Then the binary relation ``to be finer than'' is a partial order on \(\mathbf{\Pi}(X)\).
\end{proposition}

\begin{remark}\label{r1.10}
The partially ordered set \((\mathbf{\Pi}(X), \leqslant_{X})\) of all partitions of \(X\) is a complete lattice (see Section~13.3 of~\cite{Andrews1984}). The theory of lattices of partitions of a given set was developed by Oystein Ore in~\cite{Ore1942}. In particular, Ore has characterized the lattices which are isomorphic to lattice \((\mathbf{\Pi}(X), \leqslant_{X})\) for some set \(X\). It was shown by Philip~M.~Whitman~\cite{Whitman1946} that every lattice is isomorphic to a sublattice of \((\mathbf{\Pi}(X), \leqslant_{X})\) for a suitable \(X\).
\end{remark}

The following lemma is straightforward.

\begin{lemma}\label{l1.11}
Let \(X\) be a nonempty set and let \(P_1\), \(P_2 \in \mathbf{\Pi}(X)\). Then the following statements are equivalent.
\begin{enumerate}
\item\label{l1.11:s1} \(P_1 \leqslant_X P_2\).
\item\label{l1.11:s2} If \(\Psi_1 \colon X \to Y_1\) and \(\Psi_2 \colon X \to Y_2\) are surjective mappings such that
\[
P_1 = P_{\Psi_{1}^{-1}} \quad \text{and} \quad P_2 = P_{\Psi_{2}^{-1}}
\]
(see~\eqref{e1.2}), then there is a surjective mapping \(\Phi \colon Y_1 \to Y_2\) for which
\[
\Psi_2(x) = \Phi(\Psi_1(x))
\]
holds for every \(x \in X\).
\item\label{l1.11:s3} Each block of \(P_2\) is a union of some blocks of \(P_1\).
\end{enumerate}
\end{lemma}

\begin{definition}\label{d1.12}
Let \(X\) be a nonempty set, let \(R\) be an equivalence relation on \(X\) and let \(\Phi\) be a mapping with domain \(X^2\). The mapping \(\Phi\) is \(R\)-\emph{coherent} if the implication
\begin{equation}\label{d1.12:e1}
\bigl(\<x_1, x_3> \in R \text{ and } \<x_2, x_4> \in R\bigr) \Rightarrow \bigl(\Phi(x_1, x_2) = \Phi(x_3, x_4)\bigr)
\end{equation}
is valid for all \(x_1\), \(x_2\), \(x_3\), \(x_4 \in X\).
\end{definition}

For a given set \(X\) and a given equivalence relation \(R\) on \(X\), we denote by \(\Coh(R)\) the class of all surjective, \(R\)-coherent mappings with domain \(X^{2}\).

Now we recall a notion of the \emph{smallest element} of a subset of a partially ordered set. Let \((Y, \preccurlyeq)\) be a partially ordered set and let \(A \subseteq Y\). An element \(a^* \in A\) is called to be the smallest element of \(A\) if
\[
a^* \preccurlyeq a
\]
holds for every \(a \in A\). It is easy to see that the smallest element of \(A\) is unique if it exists.

The following simple theorem characterizes \(P \otimes P\) as the smallest element of the set of all partitions of \(X^{2}\) generated by mappings belonging to \(\Coh(R_P)\).

\begin{theorem}\label{t1.13}
Let \(X\) be a nonempty set and let \(Q = \{X_j \colon j \in J\}\) be a partition of \(X\). Then the inequality
\begin{equation}\label{t1.13:e1}
Q \otimes Q \leqslant_{X^{2}} P_{F^{-1}}
\end{equation}
holds for every \(F \in \Coh(R_{Q})\) and, moreover, there is \(F_{0} \in \Coh(R_Q)\) such that the equality
\[
Q \otimes Q = P_{F_{0}^{-1}}
\]
holds.
\end{theorem}

\begin{proof}
Let \(F \in \Coh(R_Q)\). Inequality~\eqref{t1.13:e1} holds if and only if we have
\begin{equation}\label{t1.13:e3}
X_{j_1} \times X_{j_2} \subseteq F^{-1}(F(x_1, x_2))
\end{equation}
for every \(\<x_1, x_2> \in X \times X\), where \(X_{j_1} \times X_{j_2}\) is a block of \(Q \otimes Q\) such that \(\<x_1, x_2> \in X_{j_1} \times X_{j_2}\). Suppose \(\<x_3, x_4>\) is an arbitrary point of \(X_{j_1} \times X_{j_2}\). Then we have
\[
x_1, x_3 \in X_{j_1} \quad \text{and} \quad x_2, x_4 \in X_{j_2},
\]
that implies \(\<x_1, x_3> \in R_Q\) and \(\<x_2, x_4> \in R_P\). Since \(F\) is \(R_Q\)-coherent, the equality \(F(x_1, x_2) = F(x_3, x_4)\) holds. Thus, we have
\[
\<x_3, x_4> \in F^{-1}(F(x_1, x_2)).
\]
Inclusion~\eqref{t1.13:e3} follows.

Let us consider \(F_{0} \colon X^2 \to J^2\) such that
\begin{equation}\label{t1.13:e2}
\bigl(F_{0}(x, y) = \<j_1, j_2>\bigr) \Leftrightarrow \bigl(x \in X_{j_1} \text{ and } y \in X_{j_2}\bigr)
\end{equation}
is valid for every \(\<x, y> \in X^{2}\). Let \(x_1\), \(x_2\), \(x_3\), \(x_4\) belong to \(X\). If \(\<x_1, x_3> \in R_Q\) and \(\<x_2, x_4> \in R_Q\) then, by Proposition~\ref{p2}, there are \(j_1 \in J\) and \(j_2 \in J\) such that
\[
x_1, x_3 \in X_{j_1} \quad \text{and} \quad x_2, x_4 \in X_{j_2}.
\]
Using these membership relations and~\eqref{t1.13:e2} we obtain
\[
F_{0}(x_1, x_2) = \<j_1, j_2> = F_{0}(x_3, x_4).
\]
Hence, \(F_{0} \in \Coh(R_Q)\) holds. It follows from~\eqref{t1.13:e2} that
\[
Q \otimes Q = \{F_{0}^{-1}(\<j_1, j_2>) \colon \<j_1, j_2> \in J^2\}. \qedhere
\]
\end{proof}

\begin{example}\label{ex1.16}
Let \(X\) be a nonempty set and let
\[
P = P_{\Delta_{X}} = \{\{x\} \colon x \in X\}
\]
be a partition of \(X\) corresponding to the diagonal \(\Delta_{X}\) on \(X\). Then \(P \otimes P\) is a partition corresponding to the diagonal \(\Delta_{X^{2}}\) on \(X^{2}\) (see Example~\ref{ex1.5}). Hence, the inequality
\begin{equation}\label{ex1.16:e1}
P \otimes P \leqslant_{X^{2}} Q
\end{equation}
holds for every partition \(Q \in \mathbf{\Pi}(X^{2})\) and, consequently, every mapping with domain \(X^{2}\) is \(\Delta_{X}\)-coherent.
\end{example}

Recall that for every nonempty \(A \subseteq \mc{B}_{X}\) we write \(\mc{S}_{A}\) for the subsemigroup of \((\mc{B}_{X}, \circ)\) having \(A\) as a set of generators.

\begin{proposition}\label{p2.26}
If a semigroup \((\mc{H}, *)\) admits a \(d\)-transitive monomorphism \(\mc{H} \to \mc{B}_{X}\), then there are an equivalence relation \(R\) on \(X\) and a mapping \(\Phi \in \Coh(R)\) such that \((\mc{H}, *)\) and \(\mc{S}_{A}\) are isomorphic with \(A = P_{\Phi^{-1}}\).
\end{proposition}

\begin{proof}
Let \(F \colon \mc{H} \to \mc{B}_{X}\) be a \(d\)-transitive monomorphism of \((\mc{H}, *)\) and let \(A\) be a set of generators of \(\mc{H}\) such that
\[
Q := \{F(a) \colon a \in A\}
\]
is a partition of \(X^{2}\). As in Example~\ref{ex1.16}, we see that the inequality
\[
P \otimes P \leqslant_{X^{2}} Q
\]
holds with \(P = \{\{x\} \colon x \in X\}\). Now the proposition follows from Theorem~\ref{t1.13}.
\end{proof}

Let \(X\) be a nonempty set and let \(R\) be an equivalence relation on \(X\). Let us denote by \(\Coh^{1}(R)\) a subclass of mappings of the class \(\Coh(R)\) such that \(\Phi \in \Coh^{1}(R)\) if and only if \(\Phi \in \Coh(R)\) and
\begin{equation*}
\Phi(x, x) = \Phi(y, y) \text{ holds for all } x, y \in X.
\end{equation*}

Analogously to Theorem~\ref{t1.13} we can characterize \(P \otimes P^{1}\) as the smallest element of the set of all partitions of \(X^{2}\) generated by mappings belonging to \(\Coh^{1}(R_P)\).

\begin{theorem}\label{t1.15}
Let \(X\) be a nonempty set and let \(Q = \{X_j \colon j \in J\}\) be a partition of \(X\). Then the inequality
\begin{equation}\label{t1.15:e1}
Q \otimes Q^{1} \leqslant_{X^{2}} P_{F^{-1}}
\end{equation}
holds for every \(F \in \Coh^{1}(R_Q)\) and, moreover, there is \(F_1 \in \Coh^{1}(R_Q)\) such that the equality
\begin{equation}\label{t1.15:e2}
Q \otimes Q^{1} = P_{F_1^{-1}}
\end{equation}
holds.
\end{theorem}

\begin{proof}
Arguing as in the proof of inequality~\eqref{t1.13:e1}, it is easy to make sure that \eqref{t1.15:e1} is true for every \(F \in \Coh^{1}(R_{Q})\).

Write
\begin{equation}\label{t1.15:e4}
\nabla_J := J^{2} \setminus \Delta_{J} \quad \text{and} \quad J^{2, 1} :=  \nabla_J \cup \{\Delta_J\},
\end{equation}
i.e., the diagonal \(\Delta_J\) is deleted from the Cartesian square \(J^{2}\) and the single-point set \(\{\Delta_J\}\) is added to the set--theoretic difference of \(J^{2}\) and \(\Delta_J\). Let us consider \(\Psi \colon J^{2} \to J^{2, 1}\) such that
\[
\Psi\bigl(\<j_1, j_2>\bigr) =
\begin{cases}
\<j_1, j_2>, & \text{if } j_1 \neq j_2\\
\Delta_J, & \text{if } j_1 = j_2
\end{cases}
\]
and let \(F_1\) denote the composition
\[
X^{2} \xrightarrow{F_0} J^{2} \xrightarrow{\Psi} J^{2, 1},
\]
where \(F_0 \colon X^{2} \to J^{2}\) is defined by~\eqref{t1.13:e2}. Then \(F_1\) belongs to \(\Coh^{1}(R_Q)\) and~\eqref{t1.15:e2} holds.
\end{proof}

Let \(X\) be a nonempty set and let \(Q\) be a partition of the set \(X^{2}\), \(Q \in \mathbf{\Pi}(X^{2})\). We say that \(Q\) is \emph{symmetric} if the equivalence
\[
\bigl(\<x, y> \in B\bigr) \Leftrightarrow \bigl(\<y, x> \in B\bigr)
\]
is valid for each block \(B\) of \(Q\) and every \(\<x, y> \in X^{2}\). Thus \(Q\) is a symmetric partition of \(X^{2}\) if every block of \(Q\) is a symmetric binary relation on \(X\).

The following proposition shows that, for every \(P \in \mathbf{\Pi}(X)\), the partition \(P \otimes P_S\) defined by~\eqref{e1.7} is the smallest symmetric partition of \(X^{2}\) with refinement \(P \otimes P\).

\begin{proposition}\label{p1.16}
Let \(X\) be a nonempty set and let \(P = \{X_j \colon j \in J\}\) be a partition of \(X\). Then \(P \otimes P_S\) is symmetric and the inequality
\begin{equation}\label{p1.16:e1}
P \otimes P \leqslant_{X^{2}} P \otimes P_S
\end{equation}
holds and, moreover, if \(Q\) is an arbitrary symmetric partition of \(X^{2}\) such that
\begin{equation}\label{p1.16:e2}
P \otimes P \leqslant_{X^{2}} Q,
\end{equation}
then we also have
\begin{equation}\label{p1.16:e3}
P \otimes P_S \leqslant_{X^{2}} Q.
\end{equation}
\end{proposition}

\begin{proof}
It follows directly from the definition of \(P \otimes P_S\) that~\eqref{p1.16:e1} holds and \(P \otimes P_S\) is symmetric.

Suppose \(Q \in \mathbf{\Pi}(X^{2})\) is symmetric and satisfies~\eqref{p1.16:e2}. Let \(\<x, y>\) be an arbitrary point of \(X^{2}\). Then there are \(j_1\), \(j_2 \in J\) such that
\begin{equation}\label{p1.16:e4}
\<x, y> \in (X_{j_1} \times X_{j_2}) \cup (X_{j_2} \times X_{j_1}).
\end{equation}
Similarly, there is a block \(B\) of \(Q\) such that
\begin{equation}\label{p1.16:e5}
\<x, y> \in B.
\end{equation}
There is also a block of \(P \otimes P\) which contains \(\<x, y>\). Using~\eqref{p1.16:e4} we can suppose, for definiteness, that this is the block \(X_{j_1} \times X_{j_2}\),
\[
\<x, y> \in X_{j_1} \times X_{j_2}.
\]
The last membership relation, \eqref{p1.16:e5} and inequality~\eqref{p1.16:e2} imply
\begin{equation}\label{p1.16:e6}
X_{j_1} \times X_{j_2} \subseteq B.
\end{equation}
Now to prove~\eqref{p1.16:e3} it suffices to show that
\begin{equation*}
(X_{j_2} \times X_{j_1}) \subseteq B.
\end{equation*}
For every binary relation \(R \subseteq X \times X\) the converse relation \(R^{-1}\) is defined as
\[
R^{-1} := \{\<y, x> \colon \<x, y> \in R\}.
\]
A binary relation \(R\) is symmetric if and only if \(R = R^{-1}\). From~\eqref{p1.16:e6} it follows that
\[
X_{j_2} \times X_{j_1} = (X_{j_1} \times X_{j_2})^{-1} \subseteq B^{-1}.
\]
Since \(B\) is symmetric, the equality \(B = B^{-1}\) holds. The inclusion
\[
X_{j_2} \times X_{j_1} \subseteq B
\]
follows.
\end{proof}

Let \(X\) be a nonempty set. A mapping \(\Phi\) with domain \(X^{2}\) is symmetric if the equality \(\Phi(x, y) = \Phi(y, x)\) holds for all \(x\), \(y \in X\).

\begin{lemma}\label{l1.17}
Let \(X\) be a nonempty set and let \(\Phi\) be a surjective mapping with domain \(X^{2}\). Then the mapping \(\Phi\) is \emph{symmetric} if and only if \(P_{\Phi^{-1}}\) is a symmetric partition of \(X^{2}\).
\end{lemma}

\begin{proof}
It follows directly from the definitions.
\end{proof}

Let \(R\) be an equivalence relation on \(X\). Denote by \(\Coh_S(R)\) the class of all symmetric mappings \(\Phi \in \Coh(R)\). The following theorem characterizes \(P \otimes P_S\) as the smallest element of the set of all partitions of \(X^{2}\) generated by mappings from \(\Coh_S(R_P)\).

\begin{theorem}\label{t1.18}
Let \(X\) be a nonempty set and let \(Q = \{X_j \colon j \in J\}\) be a partition of \(X\). Then the inequality
\begin{equation}\label{t1.18:e1}
Q \otimes Q_S \leqslant_{X^{2}} P_{F^{-1}}
\end{equation}
holds for every \(F \in \Coh_S(R_Q)\) and, moreover, there is \(F_1 \in \Coh_S(R_Q)\) such that the equality
\begin{equation}\label{t1.18:e2}
Q \otimes Q_S = P_{F_1^{-1}}
\end{equation}
holds.
\end{theorem}

\begin{proof}
Let \(F \in \Coh_S(R_Q)\). By Theorem~\ref{t1.13}, we have the inequality
\begin{equation}\label{t1.18:e3}
Q \otimes Q \leqslant_{X^{2}} P_{F^{-1}}.
\end{equation}
Since \(F\) is a symmetric mapping, the partition \(P_{F^{-1}}\) is symmetric by Lemma~\ref{l1.17}. Using Proposition~\ref{p1.16}, we see that inequality~\eqref{t1.18:e1} follows from \eqref{t1.18:e3}.

Let
\begin{equation}\label{t1.18:e4}
J \times J_S := \bigl\{\{\<j_1, j_2>, \<j_2, j_1>\} \colon \<j_1, j_2> \in \nabla_{J}\bigr\} \cup \Delta_{J},
\end{equation}
where \(\nabla_{J} = J^{2} \setminus \Delta_{J}\). Let us consider the surjection \(\Psi \colon J^{2} \to J \times J_S\) such that
\begin{equation}\label{t1.18:e5}
\Psi(j_1, j_2) = \begin{cases}
\{\<j_1, j_1>\}, & \text{if } j_1 = j_2\\
\{\<j_1, j_2>, \<j_2, j_1>\}, & \text{if } j_1 \neq j_2.
\end{cases}
\end{equation}
Write \(F_1\) for the composition
\[
X^{2} \xrightarrow{F_0} J^{2} \xrightarrow{\Psi} J \times J_S,
\]
where \(F_0\) is defined by~\eqref{t1.13:e2}. Then \(F_1\) belongs to \(\Coh_S(R_Q)\) and equality~\eqref{t1.18:e2} holds.
\end{proof}

Let \(X\) be a nonempty set and let \(P = \{X_j \colon j \in J\}\) be a partition of \(X\). Recall that \(P \otimes P_S^{1}\) is a partition of \(X^{2}\) with the blocks \(B\) such that either \(B = R_P\) or there are distinct \(j_1\), \(j_{2} \in J\) for which
\[
B = (X_{j_1} \times X_{j_2}) \cup (X_{j_2} \times X_{j_1})
\]
holds.

The following proposition shows that \(P \otimes P_S^{1}\) is the smallest symmetric partition with refinement \(P \otimes P^{1}\).

\begin{proposition}\label{p1.19}
Let \(X\) be a nonempty set and let \(P\) be a partition of \(X\). Then \(P \otimes P_S^{1}\) is symmetric and the inequality
\begin{equation}\label{p1.19:e1}
P \otimes P^{1} \leqslant_{X^{2}} P \otimes P_S^{1}
\end{equation}
holds and, moreover, if \(Q\) is an arbitrary symmetric partition of \(X^{2}\) such that
\begin{equation}\label{p1.19:e2}
P \otimes P^{1} \leqslant_{X^{2}} Q,
\end{equation}
then we also have
\begin{equation}\label{p1.19:e3}
P \otimes P_S^{1} \leqslant_{X^{2}} Q.
\end{equation}
\end{proposition}

\begin{proof}
It follows directly from the definitions of \(P \otimes P_S^{1}\) and \(P \otimes P^{1}\) that \(P \otimes P_S^{1}\) is symmetric and~\eqref{p1.19:e1} holds.

Suppose \(Q \in \mathbf{\Pi}(X^{2})\) is symmetric and satisfies~\eqref{p1.19:e2}. Then we evidently have
\begin{equation}\label{p1.19:e4}
P \otimes P \leqslant_{X^{2}} Q.
\end{equation}
By Proposition~\ref{p1.16}, inequality~\eqref{p1.19:e4} implies the inequality
\begin{equation}\label{p1.19:e5}
P \otimes P_S \leqslant_{X^{2}} Q.
\end{equation}
Inequality~\eqref{p1.19:e3} follows from~\eqref{p1.19:e5} and \eqref{p1.19:e2} because every block of \(P \otimes P_S^{1}\) is a block of \(P \otimes P^{1}\) or a block of \(P \otimes P_S\).
\end{proof}

Let \(X\) be a nonempty set and let \(R\) be an equivalence relation on \(X\). We will denote by \(\Coh_S^{1}(R)\) the class of all symmetric mappings \(\Phi \in \Coh(R)\) satisfying the equality
\[
\Phi(x, x) = \Phi(y, y)
\]
for all \(x\), \(y \in X\). It is clear that the equality
\begin{equation}\label{e1.35}
\Coh_S^{1}(R) = \Coh_S(R) \cap \Coh^{1}(R)
\end{equation}
holds for every nonvoid X and every equivalence relation \(R\) on \(X\).

\begin{theorem}\label{t1.20}
Let \(X\) be a nonempty set and let \(Q = \{X_j \colon j \in J\}\) be a partition of \(X\). Then the inequality
\begin{equation}\label{t1.20:e1}
Q \otimes Q_S^{1} \leqslant_{X^{2}} P_{\Phi^{-1}}
\end{equation}
holds for every \(\Phi \in \Coh_S^{1}(R_Q)\) and, moreover, there is \(\Phi_1 \in \Coh_S^{1}(R_Q)\) such that the equality
\begin{equation}\label{t1.20:e2}
Q \otimes Q_S^{1} = P_{\Phi_1^{-1}}
\end{equation}
holds.
\end{theorem}

\begin{proof}
Let \(\Phi \in \Coh_S^{1}(R_P)\). Then, by~\eqref{e1.35}, we obtain that \(P_{\Phi^{-1}}\) is a symmetric partition of \(X^{2}\). Since \(\Phi \in \Coh(R_P)\) holds, Theorem~\ref{t1.15} implies
\begin{equation}\label{t1.20:e3}
Q \otimes Q^{1} \leqslant_{X^{2}} P_{\Phi^{-1}}.
\end{equation}
Inequality~\eqref{t1.20:e1} follows from \(\Phi \in \Coh_S(R_Q)\), Proposition~\ref{p1.19} and \eqref{t1.20:e3}.

Let us find \(\Phi_1 \in \Coh_S^{1}(R_Q)\) such that~\eqref{t1.20:e2} holds. Similarly \eqref{t1.15:e4} and \eqref{t1.18:e4}, we define \(J \times J_S^{1} \in \mathbf{\Pi}(X^{2})\) as
\[
J \times J_S^{1} := \nabla_{J} \cup \bigl\{\Delta_J\bigr\}.
\]
Let us consider the mapping \(\Psi_1 \colon J \times J_S \to J \times J_S^{1}\) defined such that
\[
\Psi_1(\<j, j>) = \{\Delta_J\} \quad \text{and} \quad \Psi_1(\{\<j_1, j_2>, \<j_2, j_1>\}) = \{\<j_1, j_2>, \<j_2, j_1>\}
\]
hold for every \(j \in J\) and all distinct \(j_1\), \(j_2 \in J\). Write \(\Phi_1\) for the composition
\[
X^{2} \xrightarrow{F_0} J^{2} \xrightarrow{\Psi} J \times J_S \xrightarrow{\Psi_1} J \times J_S^{1},
\]
where \(F_0\) and \(\Psi\) are defined by \eqref{t1.13:e2} and \eqref{t1.18:e5}, respectively. Then \(\Phi_1\) belongs to \(\Coh_S^{1}(R_Q)\) and satisfies equality \eqref{t1.20:e2}.
\end{proof}

The results of the present section are quite elementary and should be known to experts in the theory of relations in one form or another. Note also that these results can be naturally generalized to the case of partitions of the set \(X^{K}\) for arbitrary \(K\) with \(|K| \geqslant 2\). The partitions \(P \otimes P\), \(P \otimes P^{1}\), \(P \otimes P_S\) and \(P \otimes P_S^{1}\) of \(X^{2}\) can also be described as Cartesian products of disjoint unions of complete graphs with interpretation of \(R_P\)-coherent mappings as homomorphisms of corresponding graphs. (See, for example, \cite{Imrich2008} and \cite{HN} for some results related to Cartesian products and, respectively, morphisms of graphs.)

The algebraic structure of the subsemigroups of \(\mc{B}_{X}\) generated by \(P \otimes P\), \(P \otimes P^{1}\) and \(P \otimes P_S\),\(P \otimes P_S^{1}\) will be described in Section~3 and, respectively, Section~4 of the paper.

\section{Semigroups generated by finest partitions of Cartesian squares}

Let \((\mc{S}, *)\) be a semigroup. If \(\mc{S}\) is a single-point set, \(\mc{S} = \{e\}\), then we consider that \(e\) is the identity element of \((\mc{S}, *)\). The usual convention says that \((\mc{S}, *)\) must have at least two elements to posses a zero (see, for example, \cite{Higgins1992}).

An element \(i \in \mc{S}\) is an \emph{idempotent element} of \((\mc{S}, *)\) if
\[
i^2 = i * i = i.
\]
It is clear that the identity element \(e\) and the zero \(\theta\) are idempotents. We will say that \(e\) and \(\theta\) are the \emph{trivial idempotent elements}.

\begin{definition}\label{ch2:d7}
Let \((\mc{S}, \circ)\) and \((\mc{H}, *)\) be semigroups. A mapping \(F \colon \mc{S} \to \mc{H}\) is a \emph{homomorphism} if
\[
F(s_1 \circ s_2) = F(s_1) * F(s_2)
\]
holds for all \(s_1\), \(s_2 \in \mc{S}\). If a homomorphism is injective, then it is a \emph{monomorphism}. The bijective homomorphisms are called the \emph{isomorphisms}.

The semigroups \(\mc{S}\) and \(\mc{H}\) are \emph{isomorphic} if there is an isomorphism \(F \colon \mc{S} \to \mc{H}\).
\end{definition}

Recall that, for every nonempty set \(Q\) of binary relations on a set \(X\), we denote by \(\mc{S}_{Q}\) a subsemigroup of \(\mc{B}_{X}\) having \(Q\) as a set of generators. In particular, if \(Q\) is a partition of \(X^{2}\), then every block of \(Q\) is a binary relation on \(X\) so that we can consider the semigroup \(\mc{S}_{Q}\).

Now let \(P\) be a partition of \(X\). Then \(P \otimes P\) is a partition of \(X^{2}\) and our first goal is to describe the algebraic structure of the semigroup \(\mc{S}_{P\otimes P}\) up to isomorphism.

\begin{theorem}\label{ch2:th6}
Let \((\mc{H}, *)\) be a semigroup. The following two statements are equivalent.
\begin{enumerate}
\item\label{ch2:th6:s1} There are a nonempty set \(X\) and a partition \(P\) of \(X\) such that the semigroup \((\mc{S}_{P\otimes P}, \circ)\) is isomorphic to \((\mc{H}, *)\).
\item\label{ch2:th6:s2} The semigroup \((\mc{H}, *)\) satisfies the following conditions.
\begin{enumerate}
\item \((\mc{H}, *)\) contains a zero element \(\theta\) if \(|\mc{H}| \geqslant 2\).
\item The equality
\begin{equation}\label{ch2:th6:e1}
x * y = \theta
\end{equation}
holds for all distinct idempotent elements \(x\), \(y \in \mc{H}\).
\item If \(i_{l}\) and \(i_{r}\) are nontrivial idempotent elements of \(\mc{H}\), then there is a unique nonzero \(a \in \mc{H}\) such that
\begin{equation}\label{ch2:th6:e3}
a = i_{l} * a * i_{r}.
\end{equation}
\item If \(|\mc{H}| \geqslant 2\) holds, then for every nonzero \(a \in \mc{H}\) there is a unique pair \((i_{la}, i_{ra})\) of nontrivial idempotent elements of \(\mc{H}\) such that
\begin{equation}\label{ch2:th6:e2}
a = i_{la} * a * i_{ra}.
\end{equation}
\end{enumerate}
\end{enumerate}
\end{theorem}

\begin{proof}
\(\ref{ch2:th6:s1} \Rightarrow \ref{ch2:th6:s2}\). Let \(P = \{X_j \colon j \in J\}\) be a partition of a set \(X\) such that the semigroup \((\mc{H}, *)\) is isomorphic to \((\mc{S}_{P\otimes P}, \circ)\). We must prove that \((\mc{H}, *)\) satisfies conditions \((ii_1)\)--\((ii_4)\). Since \((\mc{H}, *)\) and \((\mc{S}_{P\otimes P}, \circ)\) are isomorphic, it suffices to show that the similar conditions hold for \((\mc{S}_{P\otimes P}, \circ)\). Let us do it.

First of all we note that conditions \((ii_1)-(ii_4)\) are satisfied if \(|P| = 1\). Suppose \(|P| \geqslant 2\) holds.

\((ii_1)\). The inequality \(|P| \geqslant 2\) implies that there are two distinct \(X_{j_1}\), \(X_{j_2} \in P\). Then we have \(X_{j_1} \times X_{j_1} \in P \otimes P\) and \(X_{j_2} \times X_{j_2} \in P\otimes P\) and \(X_{j_1} \cap X_{j_2} = \varnothing\). Consequently
\[
\mc{S}_{P\otimes P} \ni (X_{j_1} \times X_{j_1}) \circ (X_{j_2} \times X_{j_2}) = \varnothing.
\]
Thus \(\mc{S}_{P\otimes P}\) contains the empty binary relation. This is the zero element of \((\mc{S}_{P\otimes P}, \circ)\).

In order to verify the fulfillment of \((ii_2)\)--\((ii_4)\), we note that
\begin{equation}\label{ch2:th6:e4}
(X_{j_1} \times X_{j_2}) \circ (X_{j_3} \times X_{j_4}) = \begin{cases}
\varnothing, & \text{if } X_{j_2} \neq X_{j_3}\\
(X_{j_1} \times X_{j_4}), & \text{if } X_{j_2} = X_{j_3}
\end{cases}
\end{equation}
holds for all \(X_{j_1}\), \(X_{j_2}\), \(X_{j_3}\), \(X_{j_4} \in P\). Thus every element of \((\mc{S}_{P\otimes P}, \circ)\) is either empty or belongs to \(P\otimes P\).

\((ii_2)\). From~\eqref{ch2:th6:e4} it follows that every nontrivial idempotent element \(i\) of \((\mc{S}_{P\otimes P}, \circ)\) has the form \(i = X_{j} \times X_{j}\) for some \(X_{j} \in P\). Now \((ii_2)\) follows from the equality \(X_{j_1} \cap X_{j_2} = \varnothing\) which holds for all different \(X_{j_1}\), \(X_{j_2} \in P\).

\((ii_3)\). Let \(i_l\) and \(i_r\) be nontrivial idempotent elements of \((\mc{S}_{P\otimes P}, \circ)\). Then there are \(j_1\), \(j_2 \in J\) such that
\[
i_l := X_{j_1} \times X_{j_1} \quad \text{and} \quad i_r := X_{j_2} \times X_{j_2}.
\]
It is clear that~\eqref{ch2:th6:e3} holds if \(a = X_{j_1} \times X_{j_2}\). Suppose now that there is a nonzero \(b \in \mc{H}\) such that
\begin{equation}\label{ch2:th6:e6}
b = i_l * b * i_r.
\end{equation}
Then we can find \(j_3\), \(j_4 \in J\) such that \(b = X_{j_3} \times X_{j_4}\). This equality and \eqref{ch2:th6:e6} give us
\begin{multline*}
b = (X_{j_3} \times X_{j_3}) * b * (X_{j_4} \times X_{j_4}) \\
= (X_{j_1} \times X_{j_1}) * (X_{j_3} \times X_{j_3}) * b * (X_{j_4} \times X_{j_4}) * (X_{j_2} \times X_{j_2}) \neq \varnothing.
\end{multline*}
Using \((ii_2)\), we have \(X_{j_1} = X_{j_3}\) and \(X_{j_2} = X_{j_4}\).

\((ii_4)\). Let \(a \in \mc{H}\) be nonzero. It was shown above that there are \(X_{j_1}\), \(X_{j_2} \in P\) such that
\[
a = X_{j_1} \times X_{j_2}.
\]
Write
\[
i_l := X_{j_1} \times X_{j_1} \quad \text{and} \quad i_r := X_{j_2} \times X_{j_2}.
\]
Now~\eqref{ch2:th6:e2} follows from~\eqref{ch2:th6:e4}. The uniqueness of representation~\eqref{ch2:th6:e2} can be proved as above.

\(\ref{ch2:th6:s2} \Rightarrow \ref{ch2:th6:s1}\). Let \((\mc{H}, *)\) satisfy condition~\ref{ch2:th6:s2}. Let \(E = E(\mc{H})\) be the set of all idempotent elements of \(\mc{H}\) and let \(P\) be the partition of \(E\) on the one-point subsets of \(E\),
\[
P = \{\{i\} \colon i \in E\}.
\]
We claim that the semigroup \((\mc{S}_{P\otimes P}, \circ)\) is isomorphic to \((\mc{H}, *)\). Using Proposition~\ref{p1} and formula~\eqref{ch2:th6:e4} we see that every element of \((\mc{S}_{P\otimes P}, \circ)\) is either empty or has a form \(s = \{i_1\} \times \{i_2\}\) for some \(i_1\), \(i_2 \in E\). From the definition of the Cartesian product, we have the equality
\[
\{i_1\} \times \{i_2\} = \{\<i_1, i_2>\},
\]
thus
\begin{equation}\label{ch2:th6:e7}
s = \{\<i_1, i_2>\},
\end{equation}
holds, where \(\<i_1, i_2> \in E \times E\). Conditions \((ii_3)\) and \((ii_4)\) imply that there is a bijection \(F \colon \mc{S}_{P\otimes P} \to \mc{H}\) such that \(F(\varnothing) = \theta\), where \(\theta\) is the zero element of \((\mc{H}, *)\), and \(F(\{\<i_1, i_2>\}) = x\), where \(x\) is a unique nonzero element of \(\mc{H}\) such that
\begin{equation}\label{ch2:th6:e8}
x = i_1 * x * i_2.
\end{equation}
It suffices to show that \(F \colon \mc{S}_{P\otimes P} \to \mc{H}\) is an isomorphism. Let \(s_1\) and \(s_2\) belong to \(\mc{S}_{P\otimes P}\). We must show that
\begin{equation}\label{ch2:th6:e10}
F(s_1) * F(s_2) = F(s_1 \circ s_2).
\end{equation}
This equality is trivially valid if \(s_1 = \varnothing\) or \(s_2 = \varnothing\). Suppose now that \(s_1 \neq \varnothing \neq s_2\) but
\begin{equation}\label{ch2:th6:e11}
s_1 \circ s_2 = \varnothing
\end{equation}
holds. Using~\eqref{ch2:th6:e7} we can find \(i_{1,1}\), \(i_{1,2}\), \(i_{2,1}\), \(i_{2,2} \in E\) such that
\begin{equation}\label{ch2:th6:e12}
s_1 = \{\<i_{1,1}, i_{1,2}>\} \quad \text{and} \quad s_1 = \{\<i_{2,1}, i_{2,2}>\}.
\end{equation}
From~\eqref{ch2:th6:e11} it follows that \(i_{1,2} \neq i_{2,1}\). By \((ii_3)\), there are the unique nonzero \(x_1\) and \(x_2 \in \mc{H}\) such that
\begin{equation}\label{ch2:th6:e13}
x_1 = i_{1,1} * x_1 * i_{1,2}\quad \text{and} \quad x_2 = i_{2,1} * x_2 * i_{2,2}.
\end{equation}
By definition of \(F\), we have the equalities \(F(\varnothing) = \theta\) and \(F(s_i) = x_i\) for \(i = 1\), \(2\). Now using~\eqref{ch2:th6:e13}, condition \((ii_2)\) and \(i_{1,2} \neq i_{2,1}\) we obtain
\begin{multline*}
F(s_1)*F(s_2) = x_1 * x_2 = i_{1,1} * x_1 * (i_{1,2} * i_{2,1}) * x_2 * i_{2,2} \\
= i_{1,1} * x_1 * \theta * x_2 * i_{2,2} = \theta = F(\varnothing).
\end{multline*}

Suppose now that \(s_1\), \(s_2\), \(s_1 \circ s_2\) are nonzero elements of \((\mc{S}_{P\times P}, \circ)\). Then \eqref{ch2:th6:e12} and \eqref{ch2:th6:e13} hold with \(i_{1,2} = i_{2,1}\) and \(s_1 \circ s_2 = \{\<i_{1,1}, i_{2,2}>\}\). It implies
\begin{equation}\label{ch2:th6:e14}
F(s_1)*F(s_2) = (i_{1,1} * x_1 * i_{1,2}) * (i_{1,2} * x_2 * i_{2,2}) = x_1 * x_2.
\end{equation}
Since \(i_{1,1}\), \(i_{1,2}\), \(i_{2,2}\) are idempotent, from~\eqref{ch2:th6:e13} we have
\[
i_{1,1} * x_1 * i_{1,2} = i_{1,1} * x_1 \quad \text{and} \quad i_{2,1} * x_1 * i_{2,2} = x_1 * i_{2,2}.
\]
Now using \eqref{ch2:th6:e14} we obtain
\begin{equation}\label{ch2:th6:e15}
F(s_1)*F(s_2) = i_{1,1} * (x_1 * x_2) * i_{2,2}.
\end{equation}
By definition of \(F\), there is a unique nonzero \(y \in \mc{H}\) such that
\[
F(\{\<i_{1,1}, i_{2,2}>\}) = y = i_{1,1} * y * i_{2,2}.
\]
Thus
\[
F(s_1 \circ s_2) = F(s_1)*F(s_2)
\]
holds for all \(s_1\), \(s_2 \in \mc{S}_{P\otimes P}\). The implication \(\ref{ch2:th6:s2} \Rightarrow \ref{ch2:th6:s1}\) follows.
\end{proof}

Let us denote by \(\mb{H}_1\) the class of all semigroups \((\mc{H}, *)\) satisfying conditions \((ii_1) - (ii_4)\) from Theorem~\ref{ch2:th6}.

\begin{lemma}\label{l3.3}
Let \(X\) be a set and let \(P\) be a partition of \(X\) with \(|P| \geqslant 2\). Then the equality
\begin{equation}\label{l3.3:e1}
\mc{S}_{P \otimes P} = \{\varnothing\} \cup P \otimes P
\end{equation}
holds.
\end{lemma}

\begin{proof}
It follows from formula~\eqref{ch2:th6:e4}.
\end{proof}

\begin{remark}\label{r3.4}
Equality~\eqref{l3.3:e1} does not hold if \(|P| = 1\). In this case we have \(\mc{S}_{P \otimes P} = P \otimes P\).
\end{remark}

\begin{corollary}\label{c3.6}
Every semigroup \((\mc{H}, *) \in \mb{H}_1\) admits a \(d\)-transitive monomorphism \(\mc{H} \to \mc{B}_X\) with a suitable set \(X\).
\end{corollary}

\begin{proof}
By Theorem~\ref{ch2:th6} for every \((\mc{H}, *) \in \mb{H}_1\) there are \(X\) and \(P \in \mathbf{\Pi}(X)\) such that \((\mc{H}, *)\) and \((\mc{S}_{P \otimes P}, \circ)\) are isomorphic. From Lemma~\ref{l3.3} and Definition~\ref{d1.1} it follows that the identity mapping \(\Id \colon \mc{S}_{P \otimes P} \to \mc{B}_X\), \(\Id(s) = s\) for every \(s \in \mc{S}_{P \otimes P}\), is a \(d\)-transitive monomorphism for every \(P \in \mathbf{\Pi}(X)\). Consequently, if \(\Phi \colon \mc{H} \to \mc{S}_{P \otimes P}\) is an isomorphism, then the mapping
\[
\mc{H} \xrightarrow{\Phi} \mc{S}_{P \otimes P} \xrightarrow{\Id} \mc{B}_X
\]
is a \(d\)-transitive monomorphism.
\end{proof}

Analyzing the proof of Theorem~\ref{ch2:th6} and using Example~\ref{ex1.5} we obtain the next corollary.

\begin{corollary}\label{c3.5}
The following conditions are equivalent for every semigroup \(\mc{H}\).
\begin{enumerate}
\item \(\mc{H} \in \mb{H}_1\).
\item There is a nonempty set \(X\) such that \(\mc{H}\) and \(\mc{S}_{P_{\Delta_{X^2}}}\) are isomorphic, where \(P_{\Delta_{X^2}}\) is a partition of \(X^{2}\) corresponding to the diagonal on \(X^{2}\).
\end{enumerate}
\end{corollary}

The last corollary claims that a semigroup \(\mc{H}\) belongs to \(\mb{H}_1\) if and only if there is a nonempty set \(X\) such that \(\mc{H}\) is generated by set of all single-point subsets of \(X^{2}\).

\begin{corollary}\label{c2}
Let \(\mc{H} \in \mb{H}_1\) and \(\mc{S} \in \mb{H}_1\) hold. Then \(\mc{H}\) and \(\mc{S}\) are isomorphic semigroups if and only if \(|\mc{H}| = |\mc{S}|\).
\end{corollary}

Using Lemma~\ref{l3.3} and Remark~\ref{r3.4}, we obtain the following.

\begin{corollary}\label{c2.8}
If \((\mc{H}, *) \in \mb{H}_1\) is finite, then \(|\mc{H}| \neq 2\) and there is a nonnegative integer \(n\) such that \(|\mc{H}| = n^2 + 1\). Conversely, if \(n \neq 1\) is a nonnegative integer, then there is \((\mc{H}, *) \in \mb{H}_1\) such that \(|\mc{H}| = n^2 + 1\).
\end{corollary}

Condition \((ii_2)\) of Theorem~\ref{ch2:th6} implies that, for every \((\mc{H}, *) \in \mb{H}_1\), the set \(E = E(\mc{H})\) of all idempotent elements of \((\mc{H}, *)\) is a subsemigroup of \(\mc{H}\). The following proposition shows that a \(d\)-transitive monomorphism \(E(\mc{H}) \to \mc{B}_{X}\) exists if and only if \(|\mc{H}| = 1\).

\begin{proposition}\label{p3.7}
Let \((\mc{H}, *)\) be a semigroup with a zero element \(\theta\). Suppose that all elements of \((\mc{H}, *)\) are idempotent, and the equality
\begin{equation}\label{p3.7:e1}
e_1 * e_2 = \theta
\end{equation}
holds for all distinct \(e_1\), \(e_2 \in \mc{H}\). Then \((\mc{H}, *)\) does not admit any \(d\)-transitive monomorphism of the form \(\mc{H} \to \mc{B}_{X}\).
\end{proposition}

\begin{proof}
Suppose contrary that there is a \(d\)-transitive monomorphism \(\Phi \colon \mc{H} \to \mc{B}_{X}\).

Let \(A\) be a set of generators of \((\mc{H}, *)\) such that \(\{\Phi(a) \colon a \in A\}\) is a partition of \(X^{2}\). Let us define a subset \(A_1\) of the set \(A\) by the rule: a point \(a \in A\) belongs to \(A_1\) if and only if there is \(x_1 \in X\) such that
\[
\<x_1, x_1> \in \Phi(a).
\]
We claim that the equality \(A_1 = A\) holds. Indeed, suppose the set \(A \setminus A_1\) is nonempty. Let \(b \in A \setminus A_1\). Then \(\Phi(b)\) is a block of the partition \(\{\Phi(a) \colon a \in A\}\). Hence, \(\Phi(b)\) is a nonempty subset of \(X^{2}\). Let \(x_1\), \(x_2\) be points of \(X\) such that
\begin{equation}\label{p3.7:e2}
\<x_1, x_2> \in \Phi(b).
\end{equation}
Then there is \(a_1 \in A_1\) for which
\begin{equation}\label{p3.7:e3}
\<x_1, x_1> \in \Phi(a_1)
\end{equation}
holds. Since \(A_1 \cap (A \setminus A_1) = \varnothing\), we have \(a_1 \neq b\). From~\eqref{p3.7:e1}, it follows that \(a_1 * b = \theta\) and, by Definition~\ref{d1.1},
\begin{equation}\label{p3.7:e4}
\Phi(a_1 * b) = \Phi(\theta) = \varnothing.
\end{equation}
Since \(\Phi\) is a homomorphism, \(\Phi(a_1 * b) = \Phi(a_1) \circ \Phi(b)\) holds. From~\eqref{p3.7:e2} and \eqref{p3.7:e3} it follows that
\[
\<x_1, x_2> \in \Phi(a_1) \circ \Phi(b),
\]
that contradicts~\eqref{p3.7:e4}.

For every \(a \in A\) define a subset \(X_{a}\) of the set \(X\) as
\[
X_{a} := \{x \in X \colon \<x, x> \in \Phi(a)\}.
\]
Since \(\{\Phi(a) \colon a \in A\}\) is a partition of \(X^{2}\) and for every \(a \in A\) we have
\[
\Delta_{X} \cap \Phi(a) \neq \varnothing,
\]
the set \(\{X_{a} \colon a \in A\}\) is a partition of \(X\). Arguing as above, we see that the equality
\begin{equation}\label{p3.7:e5}
\Phi(a) = X_{a}^{2}
\end{equation}
holds for every \(a \in A\). The equality
\[
X^{2} = \bigcup_{a \in A} \Phi(a)
\]
and~\eqref{p3.7:e5} imply that
\[
X^{2} = \bigcup_{a \in A} X_{a}^{2}.
\]
The last equality holds if and only if \(|A| = 1\). The set \(A\) is a set of generators of \((\mc{H}, *)\). Consequently, we have \(|\mc{H}| = 1\), contrary to \(\theta \in \mc{H}\).
\end{proof}

Recall that a semigroup \((\mc{H}, *)\) is a \emph{group with zero} if \((\mc{H}, *)\) contains a zero \(\theta\) and the set \(\mc{H} \setminus \{\theta\}\) is a group with respect to the multiplication \(*\). The following proposition is almost evident.

\begin{proposition}\label{p3.8}
Let \((\mc{H}, *)\) be a group with zero. Then \((\mc{H}, *)\) does not admit any \(d\)-transitive monomorphism of the form \(\mc{H} \to \mathcal{B}_{X}\).
\end{proposition}

\begin{proof}
Let \(e\) be the identity of the group \(\mc{H} \setminus \{\theta\}\). If \(\Phi \colon \mc{H} \to \mathcal{B}_{X}\) is a \(d\)-transitive monomorphism with some nonempty set \(X\), then \(\Phi(\theta) = \varnothing\) and there is a set \(A\) of generators of \((\mc{H}, *)\) such that \(P = \{\Phi(a) \colon a \in A\}\) is a partition of \(X^2\). In particular, the equality
\begin{equation}\label{p3.8:e1}
\theta = a_1 * \ldots * a_n
\end{equation}
holds with some \(a_1\), \(\ldots\), \(a_n \in A\). Since every block \(\Phi(a)\) of \(P\) is nonempty subset of \(X^{2}\), equality \(\Phi(\theta) = \varnothing\) implies that \(A\) is a subset of \(\mc{H} \setminus \{\theta\}\). Hence, every \(a \in A\) has an inverse element \(a^{-1} \in \mc{H} \setminus \{\theta\}\). Using equality~\eqref{p3.8:e1}, we obtain
\[
\theta = a_n^{-1} * \ldots a_1^{-1} * \theta = (a_n^{-1} * \ldots a_1^{-1}) * (a_1 * \ldots * a_n) = e.
\]
Thus, \(\theta = e\) and, consequently, \(\theta \in \mc{H} \setminus \{\theta\}\), contrary to the definition.
\end{proof}

Corollary~\ref{c3.6} and Proposition~\ref{p3.7} show, in particular, that the existence of \(d\)-transitive monomorphism is not, in general, a hereditary property of semigroups. Now we consider an example of a semigroup for which this property is hereditary. Recall that a semigroup \((\mc{H}, *)\) is a \emph{right zero semigroup} if \(x*y = y\) holds for all \(x\), \(y \in \mc{H}\). The \emph{left zero semigroups} are defined in a dual way (see, for example, \cite[p.~4]{Clifford1961}).

\begin{proposition}\label{p3.11}
Let \((\mc{H}, *)\) be a right (left) zero semigroup. Then every subsemigroup \(\mc{S}\) of \((\mc{H}, *)\) admits a \(d\)-transitive monomorphism \(\mc{S} \to \mc{B}_{X}\) for a suitable set \(X\).
\end{proposition}

\begin{proof}
Suppose \((\mc{H}, *)\) is a right zero semigroup. Since every subsemigroup of \((\mc{H}, *)\) is also a right zero semigroup, it suffices to find a set \(X\) and \(d\)-transitive monomorphism \(\Phi \colon \mathcal{H} \to \mathcal{B}_X\). Write \(X := \mathcal{H}\) and
\begin{equation}\label{p3.11:e1}
P := \{X \times \{x\} \colon x \in X\}.
\end{equation}
Then \(P\) is a partition of \(X^{2}\). Let us define a mapping \(\Phi \colon \mathcal{H} \to \mathcal{B}_X\) as \(\Phi(x) = X \times \{x\}\) for every \(x \in \mathcal{H}\). Then the equalities
\[
\Phi(x*y) = \Phi(y) = X \times \{y\}
\]
and
\[
\Phi(x) \circ \Phi(y) = (X \times \{x\}) \circ (X \times \{y\}) = X \times \{y\}
\]
hold for all \(x\), \(y \in \mathcal{H}\). Hence, \(\Phi\) is a homomorphism. It is clear that \(\Phi\) is injective and \(\Phi(\mathcal{H}) = P\). Since \(\mathcal{H}\) is a set of generators of \((\mathcal{H}, *)\) and \(\mathcal{H}\) has no zero element, the mapping \(\Phi\) is \(d\)-transitive monomorphism.

For the case when \((\mathcal{H}, *)\) is a left zero semigroup it suffices to consider the partition \(\{\{x\} \times X \colon x \in X\}\) instead of the partition \(P\) defined by~\eqref{p3.11:e1}.
\end{proof}

\begin{example}\label{ex3.12}
Let \(X = \{x \in \mathbb{R} \colon 0 \leqslant x \leqslant 1\}\) and let \(P = \{X_0, X_1, X_2\}\) be a trichotomy of \(X\) defined in Example~\ref{ch2:ex11}. Write
\[
P^{r} := \{X \times X_0, X \times X_1, X \times X_2\}
\]
and
\[
P^{l} := \{X_0 \times X, X_1 \times X, X_2 \times X\}.
\]
Then \(P^{r}\) and \(P^{l}\) are partitions of \(X^{2}\) (see Figure~\ref{fig2}), and \(S_{P^{r}}\) is a right zero semigroup, and \(S_{P^{l}}\) is a left zero semigroup.
\end{example}

\begin{figure}[htb]
\begin{tikzpicture}
\def\xx{0cm};
\def\yy{0cm};
\def\dx{1cm};
\def\dy{1cm};
\draw (\xx - 0.5*\dx, \yy + 2*\dy) node {$P^{r}$};
\draw (\xx-0.1*\dx, \yy) -- (\xx, \yy) node [below left] {\(0\)} -- (\xx + \dx, \yy) node [below] {\(\frac{1}{3}\)} -- (\xx + 2*\dx, \yy) node [below] {\(\frac{2}{3}\)} -- (\xx+3*\dx, \yy) node [below] {\(1\)} -- (\xx+3.1*\dx, \yy);
\draw (\xx,\yy-0.1*\dy) -- (\xx,\yy+\dy) -- (\xx,\yy+ 2*\dy) -- (\xx, \yy+3*\dy) node [left] {\(1\)} -- (\xx, \yy+3.5*\dy);

\filldraw [fill=red, draw=black] (\xx, \yy) -- (\xx+\dx, \yy) -- (\xx+\dx, \yy+3*\dy) -- (\xx, \yy+3*\dy) -- cycle;
\filldraw [fill=yellow, draw=black] (\xx+\dx, \yy) -- (\xx+2*\dx, \yy) -- (\xx+2*\dx, \yy+3*\dy) -- (\xx+\dx, \yy+3*\dy) -- cycle;
\filldraw [fill=blue, draw=black] (\xx+2*\dx, \yy) -- (\xx+3*\dx, \yy) -- (\xx+3*\dx, \yy+3*\dy) -- (\xx+2*\dx, \yy+3*\dy)-- cycle;

\def\xx{5cm};
\draw (\xx + 3.5*\dx, \yy + 2*\dy) node {$P^{l}$};
\draw (\xx-0.1*\dx, \yy) -- (\xx, \yy) node [below left] {\(0\)} -- (\xx + \dx, \yy) -- (\xx + 2*\dx, \yy) -- (\xx+3*\dx, \yy) node [below] {\(1\)} -- (\xx+3.1*\dx, \yy);
\draw (\xx,\yy-0.1*\dy) -- (\xx,\yy+\dy) node [left] {\(\frac{1}{3}\)} -- (\xx,\yy+ 2*\dy) node [left] {\(\frac{2}{3}\)} -- (\xx, \yy+3*\dy) node [left] {\(1\)} -- (\xx, \yy+3.5*\dy);

\filldraw [fill=red, draw=black] (\xx, \yy) -- (\xx+3*\dx, \yy) -- (\xx+3*\dx, \yy+\dy) -- (\xx, \yy+\dy) -- cycle;
\filldraw [fill=yellow, draw=black] (\xx, \yy+\dy) -- (\xx+3*\dx, \yy+\dy) -- (\xx+3*\dx, \yy+2*\dy) -- (\xx, \yy+2*\dy) -- cycle;
\filldraw [fill=blue, draw=black] (\xx, \yy+2*\dy) -- (\xx+3*\dx, \yy+2*\dy) -- (\xx+3*\dx, \yy+3*\dy) -- (\xx, \yy+3*\dy)-- cycle;
\end{tikzpicture}
\caption{The partitions \(P^{r}\) and \(P^{l}\) corresponding to the trichotomy \(\{X_0, X_1, X_2\}\). Here \(X \times X_0\) (\(X_0 \times X\)) is red, \(X \times X_1\) (\(X_1 \times X\)) is yellow and \(X \times X_2\) (\(X_2 \times X\)) is blue.}
\label{fig2}
\end{figure}
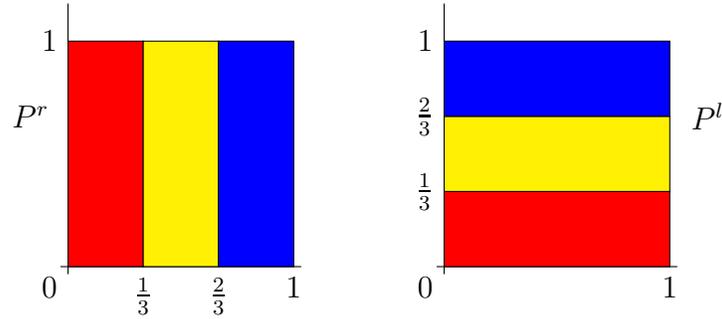

To describe the algebraic structure of the semigroup \(\mc{S}_{P\otimes P^1}\) (see~\eqref{e1.6}) we recall the procedure of ``the adjunction of an identity element''.

Let \((\mc{S}, *)\) be an arbitrary semigroup and let \(\{e\}\) be a single-point set such that \(e \notin \mc{S}\). We can extend the multiplication \(*\) from \(\mc{S}\) to \(\mc{S} \cup \{e\}\) by the rule:
\begin{equation}\label{e2.16}
e*e = e \quad \text{and} \quad e * x = x *e = x
\end{equation}
for every \(x \in \mc{S}\). Following~\cite{Clifford1961} we use the notation
\begin{equation}\label{e2.17}
\mc{S}^{1} := \begin{cases}
\mc{S}, & \text{if \((\mc{S}, *)\) has an identity element}\\
\mc{S} \cup \{e\}, & \text{otherwise}.
\end{cases}
\end{equation}
It is clear that \(e\) is an identity element of \((\mc{S}^{1}, *)\). Thus the semigroup \((\mc{S}^{1}, *)\) is obtained from \((\mc{S}, *)\) by ``adjunction of an identity element to \((\mc{S}, *)\)''.

Now we want to prove that for every nonempty set \(X\) and every partition \(P\) of \(X\), the semigroup \((\mc{S}_{P\otimes P^1}, \circ)\) can be obtained from \((\mc{S}_{P\otimes P}, \circ)\) by adjunction of an identity element.

\begin{lemma}\label{l2.10}
Let \((\mc{H}, *)\) belong to \(\mb{H}_1\). Then \((\mc{H}, *)\) contains an identity element if and only if \(|\mc{H}| = 1\).
\end{lemma}

\begin{proof}
If \(|\mc{H}| = 1\) holds, then \(\mc{H}\) contains the identity element. Suppose now \(|\mc{H}| \neq 1\). By condition~\((ii_1)\) of Theorem~\ref{ch2:th6}, \((\mc{H}, *)\) contains a zero element \(\theta\). If \(e\) is an identity element of \(\mc{H}\), then \(e \neq \theta\) holds. By condition~\((ii_4)\) of Theorem~\ref{ch2:th6}, there exists a unique pair of idempotent elements \(i_{le}\), \(i_{re} \notin \{\theta, e\}\) such that
\[
e = i_{le} * e * i_{re}
\]
holds. Now we obtain
\[
i_{le} = i_{le} * e = i_{le} * (i_{le} * e * i_{re}) = (i_{le} * i_{le}) * e * i_{re} = i_{le} * e * i_{re} = e.
\]
Thus \(i_{le} = e\) holds contrary to \(i_{le} \notin \{\theta, e\}\).
\end{proof}

\begin{theorem}\label{th2.11}
Let \((L, \cdot)\) be a semigroup. The following statements are equivalent.
\begin{enumerate}
\item\label{th2.11:s1} There are a set \(X\) and a partition \(P\) of \(X\) such that the semigroup \((\mc{S}_{P\otimes P^1}, \circ)\) is isomorphic to \((L, \cdot)\).
\item\label{th2.11:s2} There is a semigroup \((\mc{H}, *) \in \mb{H}_1\) such that \((L, \cdot)\) and \((\mc{H}^1, *)\) are isomorphic.
\end{enumerate}
\end{theorem}

Tacking into account Theorem~\ref{ch2:th6}, we obtain the following equivalent reformulation of Theorem~\ref{th2.11}.

\begin{theorem}\label{th2.12}
Let \(X\) be a set and let \(P = \{X_j \colon j \in J\}\) be a partition of \(X\). Then the semigroups \((\mc{S}_{P\otimes P}^1, \circ)\) and \((\mc{S}_{P\otimes P^1}, \circ)\) are isomorphic, where \((\mc{S}_{P\otimes P}^1, \circ)\) is obtained from \((\mc{S}_{P\otimes P}, \circ)\) by adjunction of an identity element.
\end{theorem}

\begin{proof}
The case \(|J| = 1\) is trivial.

Let \(|J| \geqslant 2\) hold. In this case, by Lemma~\ref{l2.10}, \((\mc{S}_{P\otimes P}, \circ)\) does not contain any identity element. It is easy to see that
\[
P\otimes P^{-} := \{X_{j_1} \times X_{j_2} \colon \<j_1, j_2> \in \nabla_{J}\},
\]
where \(\nabla_{J} = J^{2} \setminus \Delta_{J}\), is a set of generators of the semigroup \((\mc{S}_{P\otimes P}, \circ)\). Indeed, the equality
\[
X_{j_1}^{2} = (X_{j_1} \times X_{j_2}) \circ (X_{j_2} \times X_{j_1})
\]
holds for all \(j_1\), \(j_2 \in J\). Hence, we have \(P\otimes P \subseteq \mc{S}_{P\otimes P^{-}}\). Since \(P\otimes P\) is a set of generators of \((\mc{S}_{P\otimes P}, \circ)\) and
\[
P\otimes P^{-} \subseteq P\otimes P
\]
holds, the equality \((\mc{S}_{P\otimes P^{-}}, \circ) = (\mc{S}_{P\otimes P}, \circ)\) follows.

For every \((X_{j_1} \times X_{j_2}) \in P\otimes P\) we evidently have
\begin{align*}
(X_{j_1} \times X_{j_2}) \circ R_P &= (X_{j_1} \times X_{j_2}) \circ \left(\bigcup_{j \in J} X_j^{2} \right) \\
&= \bigcup_{j \in J} (X_{j_1} \times X_{j_2}) \circ X_j^{2} = X_{j_1} \times X_{j_2}
\end{align*}
and, similarly,
\[
R_P \circ (X_{j_1} \times X_{j_2}) = X_{j_1} \times X_{j_2}
\]
holds. Moreover, we have
\begin{align*}
R_P \circ R_P &= \left(\bigcup_{j \in J} X_j^{2}\right) \circ \left(\bigcup_{j \in J} X_j^{2}\right) \\
&= \bigcup_{i, j \in J} X_i^{2} \circ X_j^{2} = \bigcup_{j \in J} X_j^{2} = R_P
\end{align*}
and \(R_P \notin \mc{S}_{P\otimes P}\). Since \(\{R_P\} = \mc{S}_{P\otimes P^1} \setminus \mc{S}_{P\otimes P}\) and \(R_P\) is the identity element of \((\mc{S}_{P\otimes P^1}, \circ)\), the semigroup \((\mc{S}_{P\otimes P^1}, \circ)\) is obtained by adjunction of the identity element \(R_P\) to \((\mc{S}_{P\otimes P}, \circ)\).
\end{proof}

\begin{corollary}\label{c2.14}
Let \(X\) be a nonempty set and let \(P\) be a partition of \(X\) with \(|P| \geqslant 2\). Then we have
\[
\mc{S}_{P \otimes P^{1}} = \{\varnothing\} \cup \{R_P\} \cup P \otimes P.
\]
\end{corollary}

The proof of the next corollary is similar to the proof of Corollary~\ref{c3.6}.

\begin{corollary}\label{c3.17}
Let \((\mc{H}, *)\) belong to \(\mathbf{H}_{1}\). Then \((\mc{H}^{1}, *)\) admits a \(d\)-transitive monomorphism \(\mc{H}^{1} \to \mc{B}_{X}\) for a suitable set \(X\).
\end{corollary}

\section{Semigroups generated by finest symmetric partitions of Cartesian squares}

In what follows we say that a subsemigroup \(\mc{H}_1\) of a semigroup \((\mc{H}, *)\) is an \emph{ideal} of \(\mc{H}\) if
\[
\mc{H}_1 * \mc{H} \subseteq \mc{H}_1 \quad \text{and} \quad \mc{H} * \mc{H}_1 \subseteq \mc{H}_1
\]
holds, where we write
\begin{equation}\label{e2.19}
A*B := \{x*y \colon x \in A,\ y \in B\}
\end{equation}
for all nonempty subsets \(A\) and \(B\) of \(\mc{H}\). An ideal \(\mc{H}_1\) of a semigroup \(\mc{H}\) is \emph{proper} if \(|\mc{H}_1| > 1\) and \(\mc{H} \neq \mc{H}_1\) hold.

\begin{lemma}\label{l2.18}
Let \(\mc{C}\) be an ideal of a semigroup \((\mc{H}, *)\) and let \(\theta\) be the zero of \(\mc{C}\). Then \(\theta\) is also the zero of \(\mc{H}\).
\end{lemma}

\begin{proof}
Let \(a\) belong to \(\mc{H} \setminus \mc{C}\). Then \(\theta * a\) belongs to \(\mc{C}\) because \(\theta \in \mc{C}\) and \(\mc{C}\) is an ideal of \(\mc{H}\). Consequently, we have
\[
\theta = \theta * (\theta * a) = (\theta * \theta) * a = \theta * a.
\]
Similarly we obtain \(a * \theta = \theta\). Thus, \(\theta\) is a zero of \((\mc{H}, *)\).
\end{proof}

A semigroup is a \emph{band} if every element of this semigroup is idempotent. (This notion was introduced in~\cite{KL-B}.) For every \((\mc{H}, *) \in \mathbf{H}_1\) the set \(E = E(\mc{H})\) is a commutative band (this band was consider above in Proposition~\ref{p3.7}). A right (left) zero semigroup is an example of non-commutative band (see Proposition~\ref{p3.11}).

If the set \(E(\mc{H})\) of all idempotent elements of a semigroup \(\mc{H}\) is a band, then the set \(E(\mc{H}^{1})\) is also a band.

\begin{example}\label{ex2.15}
Let \(P = \{X_j \colon j \in J\}\) be a partition of a nonempty set \(X\) and let \(|J| \geqslant 2\). Then the sets \(E(\mathcal{S}_{P \otimes P^{1}})\) and \(E(\mathcal{S}_{P \otimes P})\) are commutative bands,
\[
E(\mathcal{S}_{P \otimes P^{1}}) = \{\varnothing\} \cup \{R_P\} \cup \{X_j^{2} \colon j \in J\}, \quad E(\mathcal{S}_{P \otimes P}) = \{\varnothing\} \cup \{X_j^{2} \colon j \in J\}.
\]
\end{example}

\begin{figure}[htb]
\begin{tikzpicture}[place/.style={circle,minimum size=14pt,inner sep=0pt}, node distance=1.5cm, vec/.style={-{Latex[length=3mm]}}]

\node (A1)  [place,draw=black,fill=black] {};
\node (A22) [place,draw=black,fill=yellow,below=of A1] {};
\node (A21) [place,draw=black,fill=red,left=of A22] {};
\node (A23) [place,draw=black,fill=blue,right=of A22] {};
\node (A3)  [place,draw=black,fill=white,below=of A22] {};
\draw [vec] (A1) -- ($0.3*(A1)+0.7*(A21)$); \draw ($0.3*(A1)+0.7*(A21)$) -- (A21);
\draw [vec] (A1) -- ($0.3*(A1)+0.7*(A22)$); \draw ($0.3*(A1)+0.7*(A22)$) -- (A22);
\draw [vec] (A1) -- ($0.3*(A1)+0.7*(A23)$); \draw ($0.3*(A1)+0.7*(A23)$) -- (A23);

\draw [vec] (A21) -- ($0.3*(A21)+0.7*(A3)$); \draw ($0.3*(A21)+0.7*(A3)$) -- (A3);
\draw [vec] (A22) -- ($0.3*(A22)+0.7*(A3)$); \draw ($0.3*(A22)+0.7*(A3)$) -- (A3);
\draw [vec] (A23) -- ($0.3*(A23)+0.7*(A3)$); \draw ($0.3*(A23)+0.7*(A3)$) -- (A3);
\end{tikzpicture}
\caption{\(R_P\) is white, \(X_0^{2}\) is red, \(X_1^{2}\) is yellow, \(X_2^{2}\) is blue, and \(\varnothing\) is black.}
\label{fig2.2}
\end{figure}

Every commutative band \((E, *)\) has a natural partial order \(\leqslant\) defined by
\begin{equation}\label{e2.21}
(i_2 \leqslant i_1) \Leftrightarrow (i_1 * i_2 = i_1).
\end{equation}
A colored Hasse diagram of \((E, \leqslant)\) is plotted in Figure~\ref{fig2.2} for the case when \(E\) is the band of all idempotents of \(\mc{S}_{P \otimes P^{1}}\) and \(P = \{X_0, X_1, X_2\}\) is the trichotomy introduced in Example~\ref{ch2:ex11}. A standard definition of the Hasse diagram for finite partially ordered sets can be found in~\cite[p.~15]{Howie2003}.

\begin{definition}\label{d2.15}
Let \((\mc{H}, *)\) be a semigroup and let \(\mc{C}\) be an ideal of \((\mc{H}, *)\). The semigroup \((\mc{H}, *)\) is a \emph{band of subsemigroups with core} \(\mc{C}\) if there is a partition \(\{\mc{H}_{\alpha} \colon \alpha \in \Omega\}\) of the set \(\mc{H} \setminus \mc{C}\) such that every \(\mc{H}_{\alpha}\) is a subsemigroup of \(\mc{H}\) and \(\mc{H}_{\alpha_1} * \mc{H}_{\alpha_2} \subseteq \mc{C}\) holds for all distinct \(\alpha_1\), \(\alpha_2 \in \Omega\).
\end{definition}

If \(\mc{H}\) is a band of semigroups with core \(\mc{C}\), then we write
\[
\mc{H} \approx \{\mc{H}_{\alpha}\colon \alpha \in \Omega\} \sqcup \{\mc{C}\}.
\]

\begin{example}\label{ex2.19}
Let \(\{(\mc{H}_{\alpha}, *_{\alpha}) \colon \alpha \in \Omega\}\) be a nonempty set of disjoint semigroups and \((\mc{C}, \circ)\) be a semigroup with the zero element \(\theta\) and such that
\[
\mc{C} \cap \mc{H}_{\alpha} = \varnothing
\]
holds for every \(\alpha \in \Omega\). Let us define a binary operation \(*\) on
\[
\mc{H} := \mc{C} \cup \left(\bigcup_{\alpha \in \Omega} \mc{H}_{\alpha}\right)
\]
as
\begin{equation}\label{ex2.19:e1}
x * y = \begin{cases}
x \circ y, & \text{if } x, y \in \mc{C}\\
x *_{\alpha} y, & \text{if } x, y \in \mc{H}_{\alpha}, \alpha \in \Omega\\
\theta, & \text{otherwise}.
\end{cases}
\end{equation}
It is easy to prove that \(*\) is associative. Hence, \((\mc{H}, *)\) is a semigroup, and, in addition, from~\eqref{ex2.19:e1} it follows directly that
\[
\mc{H} \approx \{\mc{H}_{\alpha}\colon \alpha \in \Omega\} \sqcup \{\mc{C}\}.
\]
\end{example}

\begin{example}\label{ex2.20}
Let \((\mc{C}, \circ)\) and \((\mc{S}, \cdot)\) be disjoint semigroups and let \(\mc{H} = \mc{C} \cup \mc{S}\). Write
\begin{equation}\label{ex2.20:e1}
x*y = \begin{cases}
x \circ y, & \text{if } x, y \in \mc{C}\\
x \cdot y, & \text{if } x, y \in \mc{S}\\
x, & \text{if } x\in \mc{C} \text{ and } y \in \mc{S}\\
y, & \text{if } y\in \mc{C} \text{ and } x \in \mc{S}.
\end{cases}
\end{equation}
Then \(* \colon \mc{H} \times \mc{H} \to \mc{H}\) is an associative operation, and \((\mc{H}, *)\) is a band of semigroups with core \(\mc{C}\).
\end{example}

The defined above band of subsemigroups with given core can be considered as a special case of the \emph{union of band of semigroups} (see, for example, \cite[p.~25]{Clifford1961}). Recall that a semigroup \((\mc{H}, *)\) is a union of band of subsemigroups \(\mc{H}_{\alpha}\), \(\alpha \in \Omega\) if
\begin{equation}\label{e2.22}
P^{\mc{H}} := \{\mc{H}_{\alpha} \colon \alpha \in \Omega\}
\end{equation}
is a partition of \(\mc{H}\) and \(\mc{H}_{\alpha} * \mc{H}_{\alpha} \subseteq \mc{H}_{\alpha}\) holds for every \(\alpha \in \Omega\) and, moreover, for every pair of distinct \(\alpha\), \(\beta \in \Omega\) there is \(\gamma \in \Omega\) such that \(\mc{H}_{\alpha} * \mc{H}_{\beta} \subseteq \mc{H}_{\gamma}\).

The next theorem gives us a characterization of subsemigroups of \(\mc{B}_{X}\) generated by partitions \(P \otimes P_S\) of \(X^{2}\) (see formula~\eqref{e1.7} and Proposition~\ref{p1.16}).

In what follows we denote by \(\theta\) a zero element of semigroup \((\mc{H}, *)\).

\begin{theorem}\label{t2.13}
Let \((\mc{H}, *)\) be a semigroup and let \(E = E(\mc{H})\) be the set of all idempotent elements of \((\mc{H}, *)\). Then the following conditions~\ref{t2.13:s1} and \ref{t2.13:s2} are equivalent.
\begin{enumerate}
\item\label{t2.13:s1} There is a nonempty set \(X\) and a partition \(P\) of \(X\) such that \(|P| \geqslant 2\) and the semigroup \((\mc{H}, *)\) is isomorphic to \((\mc{S}_{P \otimes P_S}, \circ)\).
\item\label{t2.13:s2} The semigroup \((\mc{H}, *)\) is a band of semigroups \(\mc{H}_{\alpha}\) with a core \(\mc{C}\),
\[
\mc{H} \approx \{\mc{H}_{\alpha}\colon \alpha \in \Omega\} \sqcup \{\mc{C}\},
\]
and the following conditions hold.
\begin{enumerate}
\item \(\mc{C}\) belongs to \(\mathbf{H}_{1}\).
\item Every \(\mc{H}_{\alpha}\) is a group of order \(2\).
\item \(E\) is a commutative band.
\item If \(e_1\), \(e_2\) are two distinct nontrivial idempotent elements of \(\mc{C}\), then there is a unique \(e \in E\setminus \mc{C}\) such that
\begin{equation}\label{t2.13:e0}
e_1 = e_1*e \quad \text{and} \quad e_2 = e_2*e.
\end{equation}
Conversely, if \(e \in E\setminus \mc{C}\), then there are exactly two distinct nontrivial \(e_1\), \(e_2 \in \mc{C} \cap E\) such that~\eqref{t2.13:e0} holds.
\item For every \(x \in E \cap \mc{C}\) and every \(y \in \mc{H} \setminus E\) the equality \(x*y=\theta\) (\(y*x = \theta\)) holds if and only if \(x*y\) (\(y*x\)) is idempotent.
\end{enumerate}
\end{enumerate}
\end{theorem}

\begin{proof}
\(\ref{t2.13:s1} \Rightarrow \ref{t2.13:s2}\). Let \(P = \{X_j \colon j \in J\}\) be a partition of a set \(X\) with \(|P| \geqslant 2\) and let \((\mc{S}_{P\otimes P_S}, \circ)\) be isomorphic to \((\mc{H}, *)\). We must find a proper ideal \(\mc{C}\) of \((\mc{H}, *)\) such that \(\mc{C} \in \mathbf{H}_{1}\) and prove that \(\mc{H}\) is a band with core \(\mc{C}\) satisfying conditions \((ii_2)-(ii_5)\). Since \((\mc{H}, *)\) and \((\mc{S}_{P\otimes P_S}, \circ)\) are isomorphic and \(\mc{S}_{P \otimes P}\) belongs to \(\mathbf{H}_1\), it suffices to show that \(\mc{S}_{P \otimes P_S}\) is a band with core \(\mc{S}_{P \otimes P}\) and that conditions \((ii_2) - (ii_5)\) hold with \(\mc{H} = \mc{S}_{P \otimes P_S}\) and \(\mc{C} = \mc{S}_{P \otimes P}\).

Suppose first \(|J| = 2\). Then we have \(P = \{X_1, X_2\}\) and
\[
P \otimes P = \{X_1^{2}, X_2^{2}, X_1 \times X_2, X_2 \times X_1\}
\]
and
\[
P \otimes P_S = \{X_1^{2}, X_2^{2}, (X_1 \times X_2) \cup (X_2 \times X_1)\}.
\]
Write for short
\[
a_{1,1} = X_1^{2}, \quad a_{2,2} = X_2^{2}, \quad a_{1,2} = (X_1 \times X_2) \cup (X_2 \times X_1).
\]
In this notation we obtain
\begin{equation}\label{t2.13:e1}
X_1 \times X_2 = a_{1,2} \circ a_{2,2}, \ X_2 \times X_1 = a_{1,2} \circ a_{1,1}  \text{ and } \varnothing = a_{1,1} \circ a_{2,2}.
\end{equation}
Lemma~\ref{l3.3} and equalities~\eqref{t2.13:e1} imply
\begin{equation}\label{t2.13:e2}
\mc{S}_{P \otimes P} = \{a_{1,1}, a_{2,2}, a_{1,2} \circ a_{2,2}, a_{1,2} \circ a_{1,1}, a_{1,1} \circ a_{2,2}\}.
\end{equation}
Thus \(\mc{S}_{P \otimes P}\) is a subsemigroup of \(\mc{S}_{P \otimes P_S}\). We notice that \(P \otimes P_S\) is a set of generators of \(\mc{S}_{P \otimes P_S}\), and \(P \otimes P\) is a set of generators of \(\mc{S}_{P \otimes P}\), and \(a_{1,2}\) is the unique element of \((P \otimes P_S) \setminus (P \otimes P)\). Consequently, the subsemigroup \(\mc{S}_{P \otimes P}\) of \(\mc{S}_{P \otimes P_S}\) is an ideal of \(\mc{S}_{P \otimes P_S}\) if and only if
\[
a_{1,2} \circ x \in \mc{S}_{P \otimes P} \quad \text{and} \quad x \circ a_{1,2}  \in \mc{S}_{P \otimes P}
\]
hold for every \(x \in P \otimes P\). If \(x = a_{1,1}\) or \(x = a_{2,2}\), then \(a_{1,2} \circ x \in \mc{S}_{P \otimes P}\) follows from~\eqref{t2.13:e2}. If
\[
x = X_2 \times X_1 = a_{1,2} \circ a_{1,1},
\]
then, using the equality
\begin{equation}\label{t2.13:e4}
a_{1,2} \circ a_{1,2} = a_{1,1} \cup a_{2,2}
\end{equation}
we obtain
\begin{multline}\label{t2.13:e3}
a_{1,2} \circ x = a_{1,2} \circ (a_{1,2} \circ a_{1,1}) = (a_{1,2} \circ a_{1,2}) \circ a_{1,1} \\
= (a_{1,1} \cup a_{2,2}) \circ a_{1,1} = a_{1,1},
\end{multline}
that implies \(a_{1,2} \circ x \in \mc{S}_{P \otimes P}\). Thus \(a_{1,2} \circ x \in \mc{S}_{P \otimes P}\) holds for every \(x \in P \otimes P\). Analogously we can prove that \(x \circ a_{1,2} \in \mc{S}_{P \otimes P}\) is valid for every \(x \in P \otimes P\). Thus \(\mc{S}_{P \otimes P}\) is an ideal of \(\mc{S}_{P \otimes P_S}\). This ideal is proper because we have \(a_{1,2} \in \mc{S}_{P \otimes P_S} \setminus \mc{S}_{P \otimes P}\) and \(\mc{S}_{P \otimes P}\) is single-point if and only if \(|P| = 1\).

A similar proof shows that for every \(P = \{X_j \colon j \in J\}\) with \(|J| \geqslant 3\) the semigroup \(\mc{S}_{P \otimes P}\) is a proper ideal of \(\mc{S}_{P \otimes P_S}\).

Let us prove that \(\mc{S}_{P \otimes P_S}\) is a band with core \(\mc{S}_{P \otimes P}\) and verify conditions \((ii_2)-(ii_5)\).

Write for short
\[
a_{i, j} = (X_i \times X_j) \cup (X_j \times X_i)
\]
for all \(i\), \(j \in J\). Let \(a\) belong to \(\mc{S}_{P \otimes P_S} \setminus \mc{S}_{P \otimes P}\). Since \(\mc{S}_{P \otimes P}\) is an ideal of \(\mc{S}_{P \otimes P_S}\), the element \(a\) has a form
\[
a = a_{i_1, j_1} \circ a_{i_2, j_2} \circ \ldots \circ a_{i_n, j_n},
\]
where \(n\) is a positive integer number and
\begin{equation}\label{t2.13:e20.1}
a_{i_k, j_k} \in (P \otimes P_S) \setminus (P \otimes P)
\end{equation}
for every \(k \in \{1, \ldots, n\}\). We claim that \(a\) is an element of the cyclic semigroup
\[
\langle\langle a_{i_1, j_1} \rangle\rangle = \{a_{i_1, j_1}, a_{i_1, j_1}^2, a_{i_1, j_1}^3, \ldots\}.
\]
It is clear if \(n = 1\). Let us consider the case \(n \geqslant 2\). Condition~\eqref{t2.13:e20.1} implies \(i_k \neq j_k\) for every \(k \in \{1, \ldots, n\}\). It is easy to prove that
\begin{equation}\label{t2.13:e21}
a_{i_1, j_1} \circ a_{i_2, j_2} = \varnothing
\end{equation}
holds if \(\{i_1, j_1\} \cap \{i_2, j_2\} = \varnothing\). Consequently, \(a \in \mc{S}_{P \otimes P_S} \setminus \mc{S}_{P \otimes P}\) implies
\[
|\{i_1, j_1\} \cap \{i_2, j_2\}| \geqslant 1.
\]
Let
\begin{equation}\label{t2.13:e11}
|\{i_1, j_1\} \cap \{i_2, j_2\}| = 1.
\end{equation}
Without loss of generality we can set \(j_1 = j_2\). Then
\begin{equation*}
i_1 \neq j_1 \neq i_2 \neq i_1
\end{equation*}
and
\begin{multline}\label{t2.13:e13}
a_{i_1, j_1} \circ a_{i_2, j_2} = a_{i_1, j_1} \circ a_{j_1, i_2} \\
= \Bigl((X_{i_1} \times X_{j_1}) \circ \bigl((X_{j_1} \times X_{i_2}) \cup (X_{i_2} \times X_{j_1})\bigr)\Bigr) \\
\cup \Bigl((X_{j_1} \times X_{i_1}) \circ \bigl((X_{j_1} \times X_{i_2}) \cup (X_{i_2} \times X_{j_1})\bigr)\Bigr) \\
= (X_{i_1} \times X_{i_2}) \cup \varnothing = X_{i_1} \times X_{i_2}
\end{multline}
hold. Hence, if we have \eqref{t2.13:e11}, then \(a \in \mc{S}_{P \otimes P}\) holds, contrary to \(a \in \mc{S}_{P \otimes P_S} \setminus \mc{S}_{P \otimes P}\). Let us consider the case when
\[
|\{i_1, j_1\} \cap \{i_2, j_2\}| = 2.
\]
The last equality holds if and only if \(\{i_1, j_1\} = \{i_2, j_2\}\). In this case we obtain
\begin{multline*}%
a_{i_1, j_1} \circ a_{i_2, j_2} = a_{i_1, j_1} \circ a_{i_1, j_1} = a_{i_1, j_1}^2 \\
= \bigl((X_{i_1} \times X_{j_1}) \cup (X_{j_1} \times X_{i_1})\bigr) \circ \bigl((X_{i_1} \times X_{j_1}) \cup (X_{j_1} \times X_{i_1})\bigr) = X_{i_1}^{2} \cup X_{j_1}^{2}.
\end{multline*}
From Lemma~\ref{l3.3} it follows that
\[
a_{i_1, j_1} \notin \mc{S}_{P \otimes P} \quad \text{and} \quad a_{i_1, j_1}^2 \notin \mc{S}_{P \otimes P}
\]
if \(i_1 \neq j_1\). Similarly, it can be shown that
\[
\<i_1, j_1> = \<i_3, j_3>, \quad \ldots, \quad \<i_1, j_1> = \<i_n, j_n>.
\]
The membership relation \(a \in \<\<a_{i_1, j_1}>>\) follows.

It is easy to prove that every cyclic semigroup
\[
\<\<a_{i, j}>> = \{a_{i, j}, a_{i, j}^2, a_{i, j}^3, \ldots\}
\]
is a group of order \(2\). Indeed, we have
\[
a_{i, j} = (X_{i} \times X_{j}) \cup (X_{j} \times X_{i}) \neq X_{i}^2 \cup X_{j}^2 = a_{i, j}^2
\]
and
\begin{multline*}
a_{i, j}^3 = a_{i, j}^2 \circ a_{i, j} \\
= \bigl(X_{i}^{2} \cup X_{j}^{2}\bigr) \circ \bigl((X_{i} \times X_{j}) \cup (X_{j} \times X_{i})\bigr)\\
= (X_{i} \times X_{j}) \cup (X_{j} \times X_{i}) = a_{i, j},
\end{multline*}
i.e., \(a_{i, j}^3 = a_{i, j}\). The last equality implies
\[
(a_{i, j}^2)^2 = a_{i, j}^4 = a_{i, j}^3 \circ a_{i, j} = a_{i, j} \circ a_{i, j} = a_{i, j}^2.
\]
Thus \(\<\<a_{i, j}>> = \{a_{i, j}, a_{i, j}^2\}\) is a group of order \(2\) with the identity element \(a_{i, j}^2\).

Suppose now that \(x_1\), \(x_2 \in \mc{S}_{P \otimes P_S} \setminus \mc{S}_{P \otimes P}\). If there is \(\<i, j> \in J^{2}\), \(i \neq j\), such that \(x_1 \in \<\<a_{i,j}>>\) and \(x_2 \in \<\< a_{i,j}>>\), then \(x_1 \circ x_2 \in \langle a_{i,j}\rangle\) holds because \(\<\<a_{i,j}>>\) is a group. If we have
\[
x_1 \in \<\<a_{i_1,j_1}>> \quad \text{and} \quad x_2 \in \<\<a_{i_2,j_2}>>
\]
and \(\{i_1, j_1\} \neq \{i_2, j_2\}\), then there are integer \(m \geqslant 2\) and \(n \geqslant 2\) such that
\[
x_1 \circ x_2 = a_{i_1,j_1}^m \circ a_{i_2,j_2}^n = a_{i_1,j_1}^{m-1} \circ (a_{i_1,j_1} \circ a_{i_2,j_2}) \circ a_{i_2,j_2}^{n-1}.
\]
If \(|\{i_1, j_1\} \cap \{i_2, j_2\}| = 0\), then \(a_{i_1,j_1} \circ a_{i_2,j_2} = \varnothing\) holds and, moreover, if we have \(|\{i_1, j_1\} \cap \{i_2, j_2\}| = 1\), then, as in~\eqref{t2.13:e21}, \eqref{t2.13:e13}, we obtain
\[
a_{i_1,j_1} \circ a_{i_2,j_2} \in \mc{S}_{P \otimes P}.
\]
Since \(\mc{S}_{P \otimes P}\) is an ideal of \(\mc{S}_{P \otimes P_S}\), it follows that \(x_1 \circ x_2 \in \mc{S}_{P \otimes P}\). Note now that equality \(\<\<a_{i, j}>> = \{a_{i, j}, a_{i, j}^2\}\) implies
\[
\<\<a_{i_1,j_1}>> \cap \<\<a_{i_2,j_2}>> = \varnothing
\]
if \(\{i_1, j_1\} \neq \{i_2, j_2\}\). Thus \(\mc{S}_{P \otimes P_S}\) is a band of groups \(\<\< a_{i,j}>>\) with core \(\mc{S}_{P \otimes P}\) and condition \((ii_2)\) holds.

\((ii_3)\). In what follows we denote by \(E = E(\mc{S}_{P \otimes P_S})\) the set of all idempotent elements of \(\mc{S}_{P \otimes P_S}\). It suffices to show that
\begin{equation}\label{t2.13:e16}
e_1 \circ e_2 = e_2 \circ e_1 \quad \text{and} \quad e_1 \circ e_2 \in E
\end{equation}
hold for all \(e_1\), \(e_2 \in E\). First of all we note that~\eqref{t2.13:e16} is trivially valid if \(e_1 = \varnothing\) or \(e_2 = \varnothing\) or if \(e_1\) and \(e_2\) are idempotent in \(\mc{S}_{P \otimes P}\).

Let
\[
\varnothing \neq e_1 \in \mc{S}_{P \otimes P} \cap E \quad \text{and} \quad e_2 \in (\mc{S}_{P \otimes P_{S}} \setminus \mc{S}_{P \otimes P}) \cap E.
\]
Then there are \(j_1\), \(j_2\), \(j_3 \in J\) such that \(j_2 \neq j_3\), and
\begin{equation}\label{t2.13:e17}
e_1 = X_{j_1}^2, \quad \text{and} \quad e_2 = X_{j_2}^2 \cup X_{j_3}^2.
\end{equation}
Direct calculations show that
\begin{equation}\label{t2.13:e18}
e_1 \circ e_2 = e_2 \circ e_1 = \begin{cases}
e_1, & \text{if } j_1 \in \{j_2, j_3\}\\
\varnothing, & \text{if } j_1 \notin \{j_2, j_3\},
\end{cases}
\end{equation}
that implies~\eqref{t2.13:e16}. If we have
\[
\varnothing \neq e_2 \in \mc{S}_{P \otimes P} \cap E \quad \text{and} \quad e_1 \in (\mc{S}_{P \otimes P_{S}} \setminus \mc{S}_{P \otimes P}) \cap E,
\]
then~\eqref{t2.13:e16} is proved in a similar way.

Now let
\[
e_1, e_2 \in (\mc{S}_{P \otimes P_{S}} \setminus \mc{S}_{P \otimes P}) \cap E.
\]
Then there are \(i_1\), \(j_1\), \(i_2\), \(j_2 \in J\) such that \(i_1 \neq j_1\) and \(i_2 \neq j_2\) and
\[
e_1 = X_{i_1}^2 \cup X_{j_1}^2 \quad \text{and} \quad e_2 = X_{i_2}^2 \cup X_{j_2}^2.
\]
Using these equalities we obtain
\begin{equation*}
e_1 \circ e_2 = e_2 \circ e_1 = \begin{cases}
e_1, & \text{if } \{i_1, j_1\} = \{i_2, j_2\}\\
\varnothing, & \text{if } \{i_1, j_1\} \cap \{i_2, j_2\} = \varnothing\\
X_{j}^2 , & \text{if \(j\) is a unique element of } \{i_1, j_1\} \cap \{i_2, j_2\}.
\end{cases}
\end{equation*}
The last equality also implies~\eqref{t2.13:e16}. Condition~\((ii_3)\) follows.

\((ii_4)\). Let \(e_1\) and \(e_2\) be two distinct nontrivial idempotent elements of \(\mc{S}_{P \otimes P}\). Then there are \(j_1\), \(j_2 \in J\) such that
\begin{equation}\label{t2.13:e19}
e_1 = X_{j_1}^2 \quad \text{and} \quad e_2 = X_{j_2}^2.
\end{equation}
Using~\eqref{t2.13:e18} we can show that
\begin{equation}\label{t2.13:e20}
e = X_{j_1}^2 \cup X_{j_2}^2
\end{equation}
is a unique idempotent element of \(\mc{S}_{P \otimes P_{S}} \setminus \mc{S}_{P \otimes P}\) for which~\eqref{t2.13:e0} holds.

Conversely, if \(e \in E \cap (\mc{S}_{P \otimes P_{S}} \setminus \mc{S}_{P \otimes P})\), then there are distinct \(j_1\), \(j_2 \in J\) such that~\eqref{t2.13:e20} holds. For every nontrivial idempotent \(e_3 \in \mc{S}_{P \otimes P}\), the equalities \(e_3 \circ e = e_3\) and~\eqref{t2.13:e20} imply \(e_3 = X_{j_1}^2\) or \(e_3 = X_{j_2}^2\), as required.

\((ii_5)\). We want to prove that
\begin{equation}\label{t2.13:e20.2}
(x\circ y = \varnothing) \Leftrightarrow (x\circ y \in E)
\end{equation}
and
\begin{equation}\label{t2.13:e20.3}
(y\circ x = \varnothing) \Leftrightarrow (y\circ x \in E)
\end{equation}
are valid for all \(x \in E \cap \mc{C}\) and \(y \in \mc{S}_{P \otimes P_{S}} \setminus E\).

Let us prove~\eqref{t2.13:e20.2}. The implication \((x\circ y = \varnothing) \Rightarrow (x\circ y \in E)\) is trivial. In particular, this implication is valid if \(x = \varnothing\). If \(x \in E \cap \mc{C}\), and \(x \neq \varnothing\), and \(y \in \mc{S}_{P \otimes P_{S}} \setminus E\), then there are \(j_1\), \(j_2\), \(j_3 \in J\) such that \(j_2 \neq j_3\), and
\[
x = X_{j_1}^2, \quad \text{and} \quad y = (X_{j_2} \times X_{j_3}) \cup (X_{j_3} \times X_{j_2}).
\]
These equalities imply
\begin{equation}\label{t2.13:e20.4}
x\circ y = \begin{cases}
\varnothing, & \text{if } j_1 \neq \{j_2, j_3\}\\
X_{j_2} \times X_{j_3}, & \text{if } j_1 = j_2\\
X_{j_3} \times X_{j_2}, & \text{if } j_1 = j_3.
\end{cases}
\end{equation}
Since every idempotent element of \(\mc{C}\) is either trivial or has a form \(e = X_{j}^2\) for some \(j \in J\), we see that~\eqref{t2.13:e20.4} implies the converse implication
\[
(x\circ y \in E) \Rightarrow (x\circ y = \varnothing).
\]
Equivalence~\eqref{t2.13:e20.2} is valid. The similar proof shows that~\eqref{t2.13:e20.3} is also valid.

\((ii) \Rightarrow (i)\). Suppose \((\mc{H}, *)\) is a band of semigroups \(\mc{H}_{\alpha}\) with core \(\mc{C}\),
\[
\mc{H} \approx \{\mc{H}_{\alpha}\colon \alpha \in \Omega\} \sqcup \{\mc{C}\},
\]
such that \(\mc{C} \in \mathbf{H}_{1}\) and conditions \((ii_2)-(ii_5)\) hold. By Theorem~\ref{ch2:th6}, there are a nonempty set \(X\) and a partition \(P = \{X_j \colon j \in J\}\) of \(X\) such that \((\mc{C}, *)\) is isomorphic to \((\mc{S}_{P \otimes P}, \circ)\).

Let \(\Phi \colon \mc{C} \to \mc{S}_{P \otimes P}\) be an isomorphism. We want to show that there is a continuation of \(\Phi\) to an isomorphism \(\Phi_{S} \colon \mc{H} \to \mc{S}_{P \otimes P_{S}}\) and that \(|P| \geqslant 2\) holds. The construction of \(\Phi_{S}\) will be carried out in two stages.

At the first stage, we will extend \(\Phi\) to a monomorphism (injective homomorphism) \(\Phi_1 \colon \mc{C} \cup E \to \mc{S}_{P \otimes P_{S}}\). It should be noted here that \(\mc{C} \cup E\) is also a band of semigroups with core \(\mc{C}\) because we have
\[
\mc{H} \approx \{\mc{H}_{\alpha}\colon \alpha \in \Omega\} \sqcup \{\mc{C}\},
\]
and every group \(\mc{H}_{\alpha}\) contains a unique idempotent element, and, for every \(e \in \mc{H} \setminus \mc{C}\), there is a unique \(\alpha \in \Omega\) such that \(\{e\}\) is a subgroup of \((\mc{H}_{\alpha}, *)\).

In the second stage we will prove that the monomorphism \(\Phi_1\) can be extended to an isomorphism \(\Phi_{S} \colon \mc{H} \to \mc{S}_{P \otimes P_{S}}\).

\emph{Inequality} \(|P| \geqslant 2\). The core \(\mc{C}\) has at least two nontrivial idempotent elements. Indeed, \(\{\mc{H}_{\alpha}\colon \alpha \in \Omega\}\) is a partition of \(\mc{H} \setminus \mc{C}\). Consequently, \(\mc{H} \setminus \mc{C}\) and \(\Omega\) are nonempty sets (see Remark~\ref{r3}). By condition~\((ii_2)\), every \(\mc{H}_{\alpha}\) is a group. The identity element of \(\mc{H}_{\alpha}\) is an idempotent element of \(\mc{H}\) belonging to \(\mc{H} \setminus \mc{C}\). Using condition~\((ii_4)\), we see that \(\mc{C}\) contains at least two nontrivial idempotent elements as stated above. Since \(\mc{C}\) and \(\mc{S}_{P \otimes P}\) are isomorphic and all idempotent elements of \(\mc{S}_{P \otimes P}\) are trivial if \(|P| = 1\), we have \(|P| \geqslant 2\).

\emph{Monomorphism} \(\Phi_1 \colon \mc{C} \cup E \to \mc{S}_{P \otimes P_S}\). By condition \((ii_4)\), for every \(x \in E \setminus \mc{C}\), there are exactly two distinct nontrivial idempotent elements \(x_1\), \(x_2 \in \mc{C}\) such that
\begin{equation}\label{t2.13:e22}
x_1 = x_1 * x \quad \text{and} \quad x_2 = x_2 * x.
\end{equation}
Let us define a mapping \(\Phi_1 \colon \mc{C} \cup E \to \mc{S}_{P \otimes P_S}\) as follows
\begin{equation}\label{t2.13:e23}
\Phi_1 (x) = \begin{cases}
\Phi(x), & \text{if } x \in \mc{C}\\
\Phi(x_1) \cup \Phi(x_2), & \text{if } x \in E \setminus \mc{C},
\end{cases}
\end{equation}
where \(x_1\) and \(x_2\) are idempotent elements from~\eqref{t2.13:e22}. Note that if \(z\) is a nontrivial idempotent element in \(\mc{C}\), then there is a unique \(j \in J\) such that \(\Phi(z) = X_{j}^2\). Consequently, in~\eqref{t2.13:e23} we have
\begin{equation}\label{t2.13:e24}
\Phi_{1}(x) = \Phi(x_1) \cup \Phi(x_2) = X_{j_1}^2 \cup X_{j_2}^2,
\end{equation}
where \(X_{j_1}^2 = \Phi(x_1)\) and \(X_{j_2}^2 = \Phi(x_2)\). Since
\[
\Phi(\mc{C}) = \mc{S}_{P \otimes P} \subseteq \mc{S}_{P \otimes P_S}
\]
and, for every \(x \in E \setminus \mc{C}\), we have
\[
\Phi_{1}(x) \in \mc{S}_{P \otimes P_{S}} \setminus \mc{S}_{P \otimes P},
\]
\(\Phi_{1}\) really is a mapping from \(\mc{C} \cup E\) to \(\mc{S}_{P \otimes P_S}\).

The mapping \(\Phi_1\) is injective because \(\Phi\) is injective and because condition~\((ii_4)\) and equalities~\eqref{t2.13:e23}--\eqref{t2.13:e24} imply
\begin{equation*}
\Phi_1(x) \neq \Phi_1(y)
\end{equation*}
for all different \(x\), \(y \in E \setminus \mc{C}\).

Note that
\begin{equation}\label{t2.13:e24.1}
\Phi_1(\mc{C}\cup E) = \mc{S}_{P \otimes P} \cup \{X_{j}^{2} \cup X_{i}^{2} \colon i, j \in J \text{ and } i \neq j\}.
\end{equation}
Indeed, we evidently have \(\Phi_{1}(\mc{C}) = \Phi(\mc{C}) = \mc{S}_{P \otimes P}\) and
\begin{equation}\label{t2.13:e24.3}
\Phi_{1}(E \setminus \mc{C}) \subseteq \{X_{j}^{2} \cup X_{i}^{2} \colon i, j \in J \text{ and } i \neq j\}.
\end{equation}
Using condition~\((ii_4)\), for any two distinct \(x_1\), \(x_2 \in \mc{C} \cap E\), we can find \(x \in E \setminus \mc{C}\) such that~\eqref{t2.13:e22} holds. Since \(\Phi \colon \mc{C} \to \mc{S}_{P \otimes P}\) is an isomorphism, we have the equality
\begin{equation}\label{t2.13:e24.1*}
\Phi(E \cap \mc{C}) = \{X_j^2 \colon j \in J\}.
\end{equation}
Consequently, if \(i\), \(j \in J\) and \(i \neq j\), then there is \(x \in E \setminus \mc{C}\) such that
\[
\Phi_{1}(x) = \Phi(x_1) \cup \Phi(x_2) = X_j^2 \cup X_i^2.
\]
Hence, the inclusion
\begin{equation}\label{t2.13:e24.2}
\Phi_{1}(E \setminus \mc{C}) \supseteq \{X_{j}^{2} \cup X_{i}^{2} \colon i, j \in J \text{ and } i \neq j\}
\end{equation}
holds. The last inclusion, \eqref{t2.13:e24.3} and \eqref{t2.13:e24.1*} imply~\eqref{t2.13:e24.1}.

The mapping \(\Phi_1\) is a monomorphism if and only if
\begin{equation}\label{t2.13:e25}
\Phi_1(x * y) = \Phi_1(x) \circ \Phi_1(y)
\end{equation}
holds for all \(x\), \(y \in \mc{C} \cup E\). Since \(\Phi_1\) is an extension of the isomorphism \(\Phi\), equality~\eqref{t2.13:e25} is trivial for \(x\), \(y \in \mc{C}\). In addition, we have
\[
\Phi_{1}(x*x) = \Phi_{1}(x) \circ \Phi_{1}(x)
\]
for every \(x \in E \cup \mc{C}\) because, for \(x \in E \setminus \mc{C}\), from \eqref{t2.13:e24} it follows that \(\Phi_1(x)\) is an idempotent element of \(\mc{S}_{P \otimes P_S}\). Consequently, it suffices to prove equality~\eqref{t2.13:e25} in the following cases:
\begin{gather}
\label{t2.13:e26}
x \in E \setminus \mc{C} \quad \text{and} \quad y \in E \cap \mc{C},\\
\label{t2.13:e26.1}
x \in E \cap \mc{C} \quad \text{and} \quad y \in E \setminus \mc{C},\\
\label{t2.13:e27}
x \in E \setminus \mc{C} \quad \text{and} \quad y \in \mc{C} \setminus E,\\
\label{t2.13:e28}
x \in \mc{C} \setminus E \quad \text{and} \quad y \in E \setminus \mc{C},\\
\label{t2.13:e29}
x, y \in E \setminus \mc{C} \quad \text{and} \quad x \neq y.
\end{gather}

Before proceeding to the proof of equality~\eqref{t2.13:e25} for cases~\eqref{t2.13:e26}--\eqref{t2.13:e29}, we also note that this equality holds if
\begin{equation}\label{t2.13:e29.1}
x = \theta \quad \text{or} \quad y = \theta,
\end{equation}
when \(\theta\) is the zero of \(\mc{C}\). To see it, we suppose that~\eqref{t2.13:e29.1} holds. By Lemma~\ref{l2.18}, \(\theta\) also is the zero of \((\mc{H}, *)\). Thus, we have \(x * y =\theta\) that implies
\begin{equation}\label{t2.13:e29.2}
\Phi_{1}(x*y) = \Phi_{1}(\theta) =\Phi(\theta) = \varnothing.
\end{equation}
Since \(\Phi \colon \mc{C} \to \mc{S}_{P \otimes P}\) is an isomorphism and \(\Phi_{1}\) is an extension of \(\Phi\), from~\eqref{t2.13:e29.1} and \(|P| \geqslant 2\) it follows that
\[
\Phi_{1}(x) = \varnothing \quad \text{or} \quad \Phi_{1}(y) = \varnothing.
\]
Consequently,
\begin{equation}\label{t2.13:e29.3}
\Phi_{1}(x) \circ \Phi_{1}(y) = \varnothing
\end{equation}
holds. Thus, \eqref{t2.13:e29.1} implies equality \eqref{t2.13:e25}.

Further, in proving equality~\eqref{t2.13:e25}, we will always assume
\begin{equation}\label{t2.13:e29.4}
y \neq \theta
\end{equation}
and, for \(x \in E \setminus \mc{C}\), we will set \(\Phi_{1}(x) = X_{j_1}^{2} \cup X_{j_2}^{2}\), where \(j_1\) and \(j_2\) are distinct elements of \(J\) such that
\begin{equation}\label{t2.13:e29.5}
x_1 := \Phi^{-1}(X_{j_1}^{2}), \quad x_2 := \Phi^{-1}(X_{j_2}^{2}), \quad x_1 = x_1 * x, \quad x_2 = x_2 * x
\end{equation}
(see \eqref{t2.13:e22}--\eqref{t2.13:e24}).

\emph{Case~\eqref{t2.13:e26}}. Taking condition~\eqref{t2.13:e29.4} into account, we can find \(j_3 \in J\) for which
\[
\Phi_1(y) = \Phi(y) = X_{j_3}^2.
\]
Hence, the equality
\begin{equation}\label{t2.13:e30}
\Phi_1(x) \circ \Phi_1(y) = \begin{cases}
\varnothing, & \text{if } j_3 \notin \{j_1, j_2\}\\
X_{j_3}^2, & \text{if } j_3 \in \{j_1, j_2\}.
\end{cases}
\end{equation}
holds. If \(y = x_1\) or \(y = x_2\), then from~\eqref{t2.13:e29.5} it follows that
\[
y = x_1 = x_1 * x = y * x
\]
or, respectively,
\[
y = x_2 = x_2 * x = y * x.
\]
Thus
\[
\Phi(y*x) = \Phi_1(y) = \Phi(\Phi^{-1}(X_{j_3}^2)) = X_{j_3}^2
\]
holds which implies~\eqref{t2.13:e25}.

Suppose now \(x_1 \neq y \neq x_2\). Since \((\mc{H}, *)\) is a band of semigroups with core \(\mc{C}\) and \(y \in \mc{C}\), we have \(x*y \in \mc{C}\). Moreover, by condition~\((ii_3)\), we have \(x*y \in E\). Consequently, \(x*y \in E \cap \mc{C}\) holds. If \(x*y \neq y\), then, using condition~\((ii_2)\) of Theorem~\ref{ch2:th6}, we obtain
\begin{equation}\label{t2.13:e30.1}
x*y = x*y^2 = (x * y) * y = \theta.
\end{equation}
Consequently, we have
\begin{equation}\label{t2.13:e31}
\Phi_1(x*y) = \Phi(x*y) = \Phi(\theta) = \varnothing.
\end{equation}
Using~\eqref{t2.13:e30} and \eqref{t2.13:e31} we obtain~\eqref{t2.13:e25}.

\emph{Case~\eqref{t2.13:e26.1}}. This case is completely similar to the previous one.

\emph{Case~\eqref{t2.13:e27}}. From \(y \in \mc{C} \setminus E\) and \eqref{t2.13:e29.4} it follows that there is a pair \(\<j_3, j_4> \in J^{2}\) such that
\begin{equation}\label{t2.13:e32}
\Phi_1(y) = X_{j_3} \times X_{j_4} \quad \text{and} \quad j_3 \neq j_4.
\end{equation}
Now we obtain
\begin{equation}\label{t2.13:e33}
\begin{aligned}
\Phi_1(x) \circ \Phi_1(y) &= (X_{j_1}^2 \cup X_{j_2}^2) \circ (X_{j_3} \times X_{j_4})\\
& = \begin{cases}
\varnothing, & \text{if } j_3 \notin \{j_1, j_2\}\\
X_{j_3} \times X_{j_4}, & \text{if } j_3 \in \{j_1, j_2\}
\end{cases}\\
&= \begin{cases}
\varnothing, & \text{if } j_3 \notin \{j_1, j_2\}\\
\Phi_1(y), & \text{if } j_3 \in \{j_1, j_2\}.
\end{cases}
\end{aligned}
\end{equation}
From condition~\((ii_4)\) of Theorem~\ref{ch2:th6}, it follows that \(j_3 \in \{j_1, j_2\}\) holds if and only if
\[
x_1 * y = y \quad \text{or} \quad x_2 * y = y.
\]
Suppose we have \(x_1 * y = y\). From \(x_1 = x_1 * x\) and \((ii_3)\) it follows that \(x * x_1 = x_1\). Consequently,
\[
x * y = x* (x_1 * y) = (x*x_1)*y = x_1*y = y.
\]
Thus,
\begin{equation}\label{t2.13:e34}
\Phi_1(x*y) = \Phi_1(y) = \Phi(y) = \Phi_1(x) \circ \Phi_1(y)
\end{equation}
holds if \(x_1 * y = y\). Analogously, we obtain~\eqref{t2.13:e34} if \(x_2*y=y\).

Suppose now \(j_3 \notin \{j_1, j_2\}\). Write
\begin{equation}\label{t2.13:e35}
x_3 := \Phi^{-1}(X_{j_3}^2).
\end{equation}
Then \(x_3\) is an idempotent element of \(\mc{C}\), and \(\theta \neq x_3\), and
\begin{equation}\label{t2.13:e36}
x_1 \neq x_3 \neq x_2.
\end{equation}
By condition~\((ii_3)\), we have \(x * x_3 \in E\) and, in addition, \(x * x_3 \in \mc{C}\) because \(\mc{C}\) is an ideal of \((\mc{H}, *)\). Since \(E\) is a commutative band, we obtain also
\[
x_3 * x \in \mc{C} \cap E.
\]
It is clear that
\begin{equation}\label{t2.13:e37}
(x_3*x)*x = x_3*(x*x) = x_3*x.
\end{equation}
Condition \((ii_4)\) and \eqref{t2.13:e36}--\eqref{t2.13:e37} imply that \(x_3*x\) is a trivial idempotent element of \(\mc{C}\). Since \((\mc{C}, *)\) and \((\mc{S}_{P \otimes P}, \circ)\) are isomorphic and \(|P| \geqslant 2\) holds, \(\mc{C}\) contains a zero \(\theta\), and it is the unique trivial idempotent element of \((\mc{C}, *)\). It follows directly from \eqref{t2.13:e35} and \eqref{t2.13:e32} that \(y = x_3*y\). Consequently,
\[
\Phi_{1}(x*y) = \Phi_{1}(x*x_3*y) = \Phi_{1}((x*x_3)*y) = \Phi_{1}(\theta*y) = \Phi_{1}(\theta) = \varnothing.
\]
Equality~\eqref{t2.13:e25} follows.

\emph{Case} \eqref{t2.13:e28}. This case is completely similar to~\eqref{t2.13:e27}.

\emph{Case} \eqref{t2.13:e29}. Using~\eqref{t2.13:e23} and~\eqref{t2.13:e29} we can find two distinct nontrivial \(y_1\), \(y_2 \in \mc{C} \cap E\) and, consequently, two distinct \(j_3\), \(j_4 \in J\) such that
\begin{equation}\label{t2.13:e38}
y_1 = y_1 * y, \quad y_2 = y_2 * y, \quad \{j_1, j_2\} \neq \{j_3, j_4\},
\end{equation}
and
\begin{equation}\label{t2.13:e39}
\Phi_{1}(y) = \Phi_{1}(y_1) \cup \Phi_{1}(y_2), \quad \Phi_{1}(y_1) = X_{j_3}^{2}, \quad \Phi_{1}(y_2) = X_{j_4}^{2}.
\end{equation}
It should be noted here that \(\{j_1, j_2\} \neq \{j_3, j_4\}\) holds because \(\Phi_{1}\) is injective and, by~\eqref{t2.13:e29}, we have \(x \neq y\). From~\eqref{t2.13:e39} it follows that \(\Phi_{1}(y) = X_{j_3}^{2} \cup X_{j_4}^{2}\). The last equality and \eqref{t2.13:e24} imply
\begin{multline}\label{t2.13:e40}
\Phi_{1}(x) \circ \Phi_{1}(y) = (X_{j_1}^{2} \cup X_{j_2}^{2}) \circ (X_{j_3}^{2} \cup X_{j_4}^{2})\\
= \begin{cases}
\varnothing, & \text{if } \{j_1, j_2\} \cap \{j_3, j_4\} = \varnothing\\
X_{j_0}^{2}, & \text{if \(j_0\) is a unique point of } \{j_1, j_2\} \cap \{j_3, j_4\}.
\end{cases}
\end{multline}
Let \(\{j_1, j_2\} \cap \{j_3, j_4\} \neq \varnothing\). Without loss of generality, we can set
\[
j_0 = j_1 = j_3,
\]
or, an equivalent form,
\begin{equation}\label{t2.13:e41}
x_1 = y_1 = y_0,
\end{equation}
where \(y_0 = \Phi^{-1}(X_{j_0}^2)\). Equality~\eqref{t2.13:e25} evidently holds if
\[
x*y = y_0.
\]
Let us prove the last equality. From~\eqref{t2.13:e29.5}, \eqref{t2.13:e38}, \eqref{t2.13:e39}, and \eqref{t2.13:e41} it follows that
\[
y_0 = y_0 * x \quad \text{and} \quad y_0 = y_0 * y.
\]
These equalities and condition~\((ii_3)\) imply
\[
y_0 = y_0^2 = (y_0*x)*(y_0*y) = (y_0*y_0)*(x*y) = y_0*(x*y).
\]
Hence, \(y_0 = y_0*(x*y)\) holds. Since \((\mc{H}, *)\) is a band of \(\mc{H}_{\alpha}\) with core \(\mc{C}\) and \(E\) is a commutative band, \(x*y\) is an idempotent element of \(\mc{C}\). Using condition~\((ii_2)\) of Theorem~\ref{ch2:th6}, we see that \(y_0 = y_0*(x*y)\) holds if and only if \(y_0 = \theta\) or \(y_0 = x* y\). Since \(\Phi(y_0) = X_0^2 \neq \varnothing\), the equality \(y_0 = x*y\) holds.

Suppose now that \(\{j_1, j_2\} \cap \{j_3, j_4\} = \varnothing\), i.e.,
\begin{equation}\label{t2.13:e42}
\{x_1, x_2\} \cap \{y_1, y_2\} = \varnothing.
\end{equation}
It suffices to show that \(x*y = \theta\). As above, we can prove the membership relation \(x*y \in E \cap \mc{C}\). Suppose that \(x*y \neq \theta\), i.e., \(z := x*y\) is a nontrivial idempotent element of \(\mc{C}\). It implies
\begin{equation}\label{t2.13:e43}
\theta \neq z = z * (x*y) = z^2 * (x*y) = (z*x) * (z*y).
\end{equation}
By condition~\((ii_2)\) of Theorem~\ref{ch2:th6}, from~\eqref{t2.13:e43} it follows that
\[
\theta \neq z = z*x \quad \text{and} \quad \theta \neq z = z*y.
\]
Consequently, we have \(z \in \{x_1, x_2\} \cap \{y_1, y_2\}\), contrary to~\eqref{t2.13:e42}.

Thus, \(\Phi_{1} \colon \mc{C} \cup E \to \mc{S}_{P \otimes P_{S}}\) is a monomorphism.

\emph{Isomorphism} \(\Phi_{S} \colon \mc{H} \to \mc{S}_{P \otimes P_{S}}\). Let \(x \in \mc{H} \setminus (\mc{C} \cup E)\). Then there is a unique \(\alpha \in \Omega\) such that \(x \in \mc{H}_{\alpha}\). Write \(e_x = e_{\alpha}\) for the identity element of \(\mc{H}_{\alpha}\). Then \(e_{x} \in E \setminus \mc{C}\) holds and, by~\eqref{t2.13:e24}, we have
\begin{equation}\label{t2.13:e44}
\Phi_{1}(e_{x}) = X_{i}^2 \cup X_{j}^2,
\end{equation}
where \(i = i(x)\) and \(j = j(x)\) are some distinct elements of \(J\). Let us define a mapping \(\Phi_{S} \colon \mc{H} \to \mc{S}_{P \otimes P_{S}}\) as
\begin{equation}\label{t2.13:e45}
\Phi_{S}(x) = \begin{cases}
\Phi_{1}(x), & \text{if } x \in E \cup \mc{C}\\
(X_{i} \times X_{j}) \cup (X_{j} \times X_{i}), & \text{if } x \in \mc{H} \setminus (E \cup \mc{C}),
\end{cases}
\end{equation}
where \(i = i(x)\) and \(j = j(x)\) are elements of \(J\) for which~\eqref{t2.13:e44} holds. The mapping \(\Phi_{S}\) is correctly defined because \(\Phi_{1}\) is a mapping from \(E \cup \mc{C}\) to \(\mc{S}_{P \otimes P_S}\), and \(\{\mc{H}_{\alpha} \colon \alpha \in \Omega\}\) is a partition of \(\mc{H} \setminus \mc{C}\), and every \(\mc{H}_{\alpha}\) is a group, and every group contains a unique identity element.

We claim that \(\Phi_{S}\) is a bijection. Indeed, as in the proof of~\eqref{t2.13:e24.2}, we can show that for any two distinct \(i\), \(j \in J\) there is \(\alpha \in \Omega\) such that the equality
\begin{equation}\label{t2.13:e47}
\Phi_{1}(e_{\alpha}) = X_i^2 \cup X_{j}^2
\end{equation}
holds. If \(x \neq e_{\alpha}\) and \(x \in \mc{H}_{\alpha}\) hold, then from~\eqref{t2.13:e45} and \eqref{t2.13:e47} we obtain
\[
\Phi_{S}(x) = (X_i \times X_j) \cup (X_j \times X_i).
\]
Since we have the equality
\begin{multline*}
\mc{S}_{P \otimes P_{S}} = \mc{S}_{P \otimes P} \cup \{(X_i \times X_j) \cup (X_j \times X_i) \colon i, j \in J, i \neq j\} \\
\cup \{X_i^2\cup X_j^2 \colon i, j \in J, i \neq j\},
\end{multline*}
equality~\eqref{t2.13:e24.1} implies that the mapping \(\Phi_{S}\) is surjective. Moreover, \(\{X_j \colon j \in J\}\) is a partition of \(X\),
\begin{multline*}
(X_{i_1}^2 \cup X_{j_1}^2 = X_{i_2}^2 \cup X_{j_2}^2) \Leftrightarrow  (\{i_1, j_1\} =  \{i_2, j_2\})\\
\Leftrightarrow ((X_{i_1} \times X_{j_1}) \cup (X_{j_1} \times X_{i_1}) = (X_{i_2} \times X_{j_2}) \cup (X_{j_2} \times X_{i_2}))
\end{multline*}
are valid for all two-point subsets \(\{i_1, j_1\}\) and \(\{i_2, j_2\}\) of \(J\). Hence, \(\Phi_{S}\) is injective and, consequently, bijective as was claimed above.

The bijection \(\Phi_{S} \colon \mc{H} \to \mc{S}_{P \otimes P_{S}}\) is an isomorphism if and only if
\begin{equation}\label{t2.13:e48}
\Phi_{S}(x*y) = \Phi_{S}(x) \circ \Phi_{S}(y)
\end{equation}
holds for all \(x\), \(y \in \mc{H}\).

Let us prove equality~\eqref{t2.13:e48}.

First of all we note that \eqref{t2.13:e48} is equivalent to equality~\eqref{t2.13:e25} if \(x\), \(y \in E \cup \mc{C}\). Moreover, if we have~\eqref{t2.13:e29.1}, then \eqref{t2.13:e48} can be proved similarly to \eqref{t2.13:e29.2}--\eqref{t2.13:e29.3}.

In what follows we assume that~\eqref{t2.13:e29.4} holds.

Suppose~\eqref{t2.13:e48} holds if
\begin{gather}
\label{t2.13:e49}
x \in \mc{C} \quad \text{and} \quad y \in \mc{H} \setminus (E \cup \mc{C})\\
\intertext{or if}
\label{t2.13:e50}
x \in \mc{H} \setminus (E \cup \mc{C})\quad \text{and} \quad y \in \mc{C}.
\end{gather}
Then equality~\eqref{t2.13:e48} holds for all \(x\), \(y \in \mc{H}\) if and only if it holds for all \(x\), \(y \in \mc{H} \setminus \mc{C}\). Let \(t_1\) and \(t_2\) be arbitrary points of \(\mc{H} \setminus \mc{C}\). Then there are \(\alpha_1\), \(\alpha_2 \in \Omega\) such that \(t_1 \in \mc{H}_{\alpha_1}\) and \(t_2 \in \mc{H}_{\alpha_2}\). Note that, for every \(\alpha \in \Omega\), the restriction \(\Phi_{S}|_{\mc{H}_{\alpha}} \colon \mc{H}_{\alpha} \to \Phi_{S}(\mc{H}_{\alpha})\) is an isomorphism. Consequently, if \(\alpha_1 = \alpha_2\), then we have the equality
\[
\Phi_{S}(t_1 *t_2) = \Phi_{S}(t_1) \circ \Phi_{S}(t_2).
\]
In particular, we have
\begin{equation}\label{t2.13:e50.1}
\Phi_{S}(t *t^2) = \Phi_{S}(t^2 *t) = \Phi_{S}(t) \circ \Phi_{S}(t^2) = \Phi_{S}(t^2) \circ \Phi_{S}(t)
\end{equation}
for every \(t \in \mc{H}_{\alpha}\) and every \(\alpha \in \Omega\).

\emph{Suppose} \(\alpha_1 \neq \alpha_2\). Since every \(\mc{H}_{\alpha}\) is a group of order \(2\), the equalities
\[
t_1 = t_1^3 \quad \text{and} \quad t_2 = t_2^3
\]
hold. Hence, we have
\begin{equation}\label{t2.13:e51}
\Phi_{S}(t_1 *t_2) = \Phi_{S}(t_1^3 *t_2^3) = \Phi_{S}(t_1 * (t_1^2 *t_2^2)* t_2).
\end{equation}
Since \(\mc{C}\) is a core of \(\mc{H}\), the condition \(\alpha_1 \neq \alpha_2\) implies \(t_1^3 *t_2^2 \in \mc{C}\), and \(t_1^2 *t_2^3 \in \mc{C}\), and \(t_1^2 *t_2^2 \in \mc{C}\). Consequently, from~\eqref{t2.13:e51} and our supposition it follows that
\begin{multline}\label{t2.13:e52}
\Phi_{S}(t_1 *t_2) = \Phi_{S}(t_1 * (t_1^2*t_2^2)) \circ \Phi_{S}(t_2) \\
= \Phi_{S}(t_1) \circ \Phi_{S}(t_1^2 *t_2^2) \circ \Phi_{S}(t_2).
\end{multline}
The elements \(t_1^2\), \(t_2^2\), and \(t_1^2 *t_2^2\) are idempotent and, by definition of \(\Phi_{S}\), we have
\[
\Phi_{1}|_E = \Phi_{S}|_E.
\]
Hence, using~\eqref{t2.13:e50.1}, \eqref{t2.13:e52}, and \eqref{t2.13:e25}, we obtain
\begin{multline*}
\Phi_{S}(t_1 *t_2) = (\Phi_{S}(t_1) \circ \Phi_{S}(t_1^2)) \circ (\Phi_{S}(t_2^2) \circ \Phi_{S}(t_2)) \\
= \Phi_{S}(t_1^3) \circ \Phi_{S}(t_2^3) = \Phi_{S}(t_1) \circ \Phi_{S}(t_2).
\end{multline*}
Consequently, it suffices to prove~\eqref{t2.13:e48} if \eqref{t2.13:e49}  or \eqref{t2.13:e50} holds. Notice now that instead of condition~\eqref{t2.13:e49}, we can use the stronger condition
\begin{equation}\label{t2.13:e52.1}
x \in \mc{C} \cap E \quad \text{and} \quad y \in \mc{H} \setminus (E \cup \mc{C}).
\end{equation}
Indeed, from~\((ii_4)\) of Theorem~\ref{ch2:th6} it follows that for every \(x \in \mc{C}\) there is a nontrivial idempotent \(e \in \mc{C}\) such that
\begin{equation}\label{t2.13:e53}
x = x * e.
\end{equation}
For \(x \in \mc{C}\) and \(y \in \mc{H} \setminus (E \cup \mc{C})\), equality~\eqref{t2.13:e53} implies
\begin{equation}\label{t2.13:e54}
\Phi_{S}(x * y) = \Phi_{S}(x * e * y) = \Phi_{S}((x * e)*(e * y)).
\end{equation}
Since \(x * e\) and \(e * y\) belong to \(\mc{C}\), we can rewrite~\eqref{t2.13:e54} as
\begin{equation}\label{t2.13:e55}
\Phi_{S}(x * y) = \Phi_{S}(x * e) \circ \Phi_{S}(e * y) = \Phi_{S}(x) \circ \Phi_{S}(e) \circ \Phi_{S}(e * y).
\end{equation}
It is clear that \(e \in \mc{C} \cap E\). Consequently, if \eqref{t2.13:e48} holds for all \(x\), \(y\) satisfying \eqref{t2.13:e52.1}, then \eqref{t2.13:e55} and \eqref{t2.13:e53} imply
\begin{multline*}
\Phi_{S}(x * y) = \Phi_{S}(x) \circ \Phi_{S}(e) \circ \Phi_{S}(e) \circ \Phi_{S}(y) \\
= \Phi_{S}(x * e* e) \circ \Phi_{S}(y) = \Phi_{S}(x * e) \circ \Phi_{S}(y) = \Phi_{S}(x) \circ \Phi_{S}(y).
\end{multline*}
Similarly, instead of \eqref{t2.13:e50} we may use the condition
\begin{equation}\label{t2.13:e56}
x \in \mc{H} \setminus (E \cup \mc{C}) \quad \text{and} \quad y \in \mc{C} \cap E.
\end{equation}

Let \eqref{t2.13:e52.1} hold. Then, using \eqref{t2.13:e45} and \eqref{t2.13:e29.4}, we can find \(i\), \(j\), \(k \in J\) such that \(i \neq j\) and
\begin{equation}\label{t2.13:e57}
\Phi_{S}(x) = X_k^{2} \quad \text{and} \quad \Phi_{S}(y) = (X_i \times X_j) \cup (X_j \times X_i).
\end{equation}
From \eqref{t2.13:e57} it follows that
\begin{equation}\label{t2.13:e58}
\Phi_{S}(x) \circ \Phi_{S}(y) = \begin{cases}
\varnothing, & \text{if } k \notin \{i, j\}\\
X_j \times X_i, & \text{if } k = j\\
X_i \times X_j, & \text{if } k = i.
\end{cases}
\end{equation}
If \(k \notin \{i, j\}\), then, using the equality \(y^3 = y\) and Lemma~\ref{l2.18} as in~\eqref{t2.13:e30.1}, we obtain
\[
x *y = (x * y^2)*y = \theta *y = \theta,
\]
and, consequently,
\begin{equation}\label{t2.13:e59}
\Phi_{S}(x*y) = \Phi_{S}(\theta) = \varnothing.
\end{equation}
Now~\eqref{t2.13:e48} follows from \eqref{t2.13:e58} and \eqref{t2.13:e59}.

Let \(k = j\) hold. Write
\[
e_j := \Phi_{1}^{-1}(X_j^2) \quad \text{and} \quad e_i := \Phi_{1}^{-1}(X_i^2).
\]
From the definition of \(\Phi_{1}\) we obtain the equalities
\[
e_j = \Phi^{-1}(X_j^2) \quad \text{and} \quad e_i = \Phi^{-1}(X_i^2)
\]
and, using condition~\((ii_4)\) of Theorem~\ref{ch2:th6}, prove that
\begin{equation}\label{t2.13:e60}
z = \Phi_{S}^{-1}(X_j \times X_i)
\end{equation}
holds if and only if we have \(z \in \mc{C} \setminus \{\theta\}\) and
\begin{equation}\label{t2.13:e61}
z = e_j * z * e_i.
\end{equation}
Consequently, \eqref{t2.13:e48} holds if and only if
\begin{equation}\label{t2.13:e62}
x * y = e_j *(x*y) * e_i
\end{equation}
and \(x * y \in \mc{C} \setminus \{\theta\}\). Suppose
\begin{equation}\label{t2.13:e63}
x * y = \theta
\end{equation}
holds. From \(k=j\), and \eqref{t2.13:e45}, and \eqref{t2.13:e40} it follows that \(x*y^2 = x\). Now using \eqref{t2.13:e63} and Lemma~\ref{l2.18}, we obtain
\[
x = (x*y)*y = \theta * y = \theta.
\]
Hence, \(x = \theta\) that contradicts \eqref{t2.13:e29.4}. Consequently, we have \(x * y\neq \theta\). By condition~\((ii_5)\), from \(x * y\neq \theta\) and \(x \in E \cap \mc{C}\) and \(y \in \mc{H} \setminus (E \cup \mc{C})\) it follows that \(x * y \notin E\). Moreover, \(x * y \in \mc{C}\) holds because \(x \in \mc{C}\) and \(\mc{C}\) is a core of \(\mc{H}\). Consequently, the membership relation
\begin{equation}\label{t2.13:e64}
x * y \in \mc{C}\setminus E
\end{equation}
holds. By condition~\((ii_1)\), \(\mc{C}\) belongs to \(\mathbf{H}_1\). Now, using conditions~\((ii_2)\) and \((ii_4)\) of Theorem~\ref{ch2:th6}, we obtain that there is a unique pair \(i_l\), \(i_r\) of \emph{distinct} nontrivial idempotent elements of \(\mc{C}\) such that
\begin{equation}\label{t2.13:e65}
x * y = i_l * (x*y)*i_r.
\end{equation}
Since \(x\) is also a nontrivial idempotent element of \(\mc{C}\), condition~\((ii_2)\) of Theorem~\ref{ch2:th6} implies \(i_l = x = \Phi^{-1}(X_j^2) = e_j\). Suppose \(i_r \neq e_i\). Then, using the definitions of \(\Phi_{S}\) and \(\Phi_{1}\), we obtain \(y^2 * i_r = \theta\). The last equality and \eqref{t2.13:e65} imply
\[
x * y = i_l * (x*y)*i_r = i_l * (x*y*y^2)*i_r = i_l * (x*y)*\theta = \theta,
\]
that \(x * y \neq \theta\). Consequently, \(i_r = e_i\) holds. Equality~\eqref{t2.13:e62} follows from \eqref{t2.13:e65}.

The case when condition~\eqref{t2.13:e56} holds can be analyzed similarly.

The proof of the theorem is completed.
\end{proof}

\begin{figure}[ht]
\[
\begin{array}{|c|c|c|c|c|c|c|c|}
\hline
            & \varnothing & xy          & yx          &         x^2 & y^2         & xy + yx
&   x^2 + y^2\\ \hline
\varnothing & \varnothing & \varnothing & \varnothing & \varnothing & \varnothing & \varnothing
&\varnothing \\ \hline
xy          & \varnothing & \varnothing & x^2         & \varnothing & xy          & xy
&xy          \\ \hline
yx          & \varnothing & y^2         & \varnothing & yx          & \varnothing & yx
&yx          \\ \hline
x^2         & \varnothing & xy          & \varnothing & x^2         & \varnothing & x^2
& x^2        \\ \hline
y^2         & \varnothing & \varnothing & yx          & \varnothing &  y^2         & y^2
& y^2        \\ \hline
xy + yx     & \varnothing & xy          & yx          & x^2         &  y^2         & x^2 + y^2
& xy + yx    \\ \hline
x^2 + y^2   & \varnothing & xy          & yx          & x^2         &  y^2         & xy + yx
& x^2 + y^2   \\ \hline
\end{array}
\]
\caption{The Cayley table of the disjoint union of the semigroups \(\mathcal{C} = \{\varnothing, xy, yx, x^2, y^2\}\) and \(G = \{xy + yx, x^2 + y^2\}\) with \(xy = X \times Y\), \(yx = Y \times X\), \(x^2 = X^{2}\), \(y^{2} = Y^{2}\), \(xy + yx = (X \times Y) \cup (Y \times X)\), \(x^2 + y^2 = X^2 \cup Y^2\) and \(X \cap Y = \varnothing\). Here \(\mc{C}\) is isomorphic to \(\mc{S}_{P \otimes P}\) with \(|P| = 2\), and \(G\) is a group of order \(2\), and every element of \(G\) is a two-sided identity for elements of \(\mathcal{C}\).}
\label{f2.1}
\end{figure}

\begin{remark}\label{r2.22}
Considering the semigroup \((\mc{H}, *)\) from Example~\ref{ex2.19} such that \(\mc{C} \in \mathbf{H}_{1}\) and every \(\mc{H}_{\alpha}\), \(\alpha \in \Omega\), is a group of order \(2\), we see that \(\mc{H}\) is a band of semigroups with core \(\mc{C}\) and condition~\((ii_5)\) of Theorem~\ref{t2.13} is trivially holds. Moreover, since \(E(\mc{C})\) is a commutative band, the definition of \((\mc{H}, *)\) (see \eqref{ex2.19:e1}) implies that \(E(\mc{H})\) is also a commutative band. Consequently. even if we have
\begin{equation}\label{r2.22:e1}
\mc{H} \approx \{\mc{H}_{\alpha}\colon \alpha \in \Omega\} \sqcup \{\mc{C}\},
\end{equation}
conditions \((ii_1)\), \((ii_2)\), \((ii_3)\), and \((ii_5)\) do not imply condition \((ii_4)\). Analogously, using Example~\ref{ex2.20} we can define \((\mc{H}, *)\) such that~\eqref{r2.22:e1} holds, conditions \((ii_1) - (ii_4)\) are satisfied but \((ii_5)\) is false (see Figure~\ref{f2.1} for the Caley table of corresponding \((\mc{H}, *)\)).
\end{remark}

Let us denote by \(\mathbf{H}_S\) the class of all semigroups \((\mc{H}, *)\) satisfying condition~\((ii)\) of Theorem~\ref{t2.13}.

\begin{corollary}\label{c4.9}
Every semigroup \((\mc{H}, *) \in \mathbf{H}_S\) admits a \(d\)-transitive monomorphism \(\mc{H} \to \mc{B}_{X}\) for a suitable set \(X\).
\end{corollary}

Recall that, for every semigroup \((\mc{S}, \circ)\), we denote by \((\mc{S}^{1}, \circ)\) a semigroup obtained from \((\mc{S}, \circ)\) by adjunction of an identity element (see~\eqref{e2.17}).

\begin{theorem}\label{t2.20}
Let \((L, \cdot)\) be a nonempty semigroup. The following statements are equivalent.
\begin{enumerate}
\item\label{t2.20:s1} There are a set \(X\) and a partition \(P\) of \(X\) such that the semigroups \((\mc{S}_{P \otimes P_S^1}, \circ)\) and \((L, \cdot)\) are isomorphic and \(|P| \geqslant 2\).
\item\label{t2.20:s2} There is a semigroup \((\mc{H}, *) \in \mathbf{H}_S\) such that \((L, \cdot)\) and \((\mc{H}^{1}, *)\) are isomorphic.
\end{enumerate}
\end{theorem}

The proof of this theorem is similar to the proof of Theorem~\ref{th2.12} and we omit it here.

The following corollary can be proved similarly to Corollary~\ref{c3.6}.

\begin{corollary}\label{c4.11}
Let \((\mc{H}, *)\) belong to \(\mathbf{H}_S\). Then \((\mc{H}^{1}, *)\) admits a \(d\)-transitive monomorphism \(\mc{H}^{1} \to \mc{B}_{X}\) for a suitable set \(X\).
\end{corollary}

\begin{remark}\label{r4.12}
A semigroup \((\mc{H}, *) \in \mathbf{H}_S\) has an identity element if and only if \(|\mc{H}| = 7\). (The last equality holds if and only if \((\mc{H}, *)\) is isomorphic to \((\mc{S}_{P \otimes P_S}, \circ)\) with \(|P| = 2\).)
\end{remark}

\begin{example}\label{ex2.21}
Let \(P = \{X_0, X_1, X_2\}\) be the trichotomy of the set \(X = [0,1]\) defined in Example~\ref{ch2:ex11}. Then the set
\[
E = \{\varnothing\} \cup \{R_P\} \cup \{X_0^2, X_1^2, X_2^2\} \cup \{X_0^2 \cup X_1^2, X_0^2 \cup X_2^2, X_1^2 \cup X_2^2\}
\]
is the band of all idempotents of \((\mc{S}_{P \otimes P_{S}^{1}}, \circ)\). This band is commutative and it is a lattice with respect to the partial order \(\leqslant\) defined by~\eqref{e2.21}. A colored Hasse diagram of \((E, \leqslant)\) is plotted in Figure~\ref{fig2.3}.
\end{example}

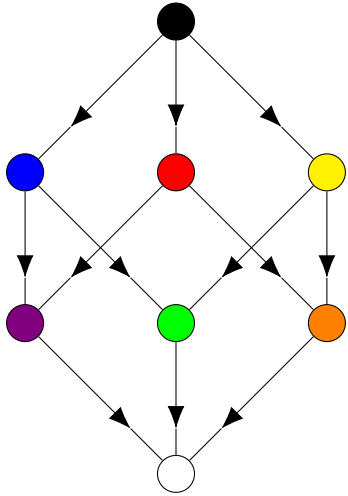
\begin{figure}[ht]
\begin{tikzpicture}[place/.style={circle,minimum size=14pt,inner sep=0pt}, node distance=1.5cm, vec/.style={-{Latex[length=3mm]}}]

\node (A1)  [place,draw=black,fill=black] {};
\node (A22) [place,draw=black,fill=red,below=of A1] {};
\node (A21) [place,draw=black,fill=blue,left=of A22] {};
\node (A23) [place,draw=black,fill=yellow,right=of A22] {};
\draw [vec] (A1) -- ($0.3*(A1)+0.7*(A21)$); \draw ($0.3*(A1)+0.7*(A21)$) -- (A21);
\draw [vec] (A1) -- ($0.3*(A1)+0.7*(A22)$); \draw ($0.3*(A1)+0.7*(A22)$) -- (A22);
\draw [vec] (A1) -- ($0.3*(A1)+0.7*(A23)$); \draw ($0.3*(A1)+0.7*(A23)$) -- (A23);

\node (A31) [place,draw=black,fill=violet,below=of A21] {};
\node (A32) [place,draw=black,fill=green,below=of A22] {};
\node (A33) [place,draw=black,fill=orange,below=of A23] {};
\node (A4)  [place,draw=black,fill=white,below=of A32] {};

\draw [vec] (A21) -- ($0.3*(A21)+0.7*(A31)$); \draw ($0.3*(A21)+0.7*(A31)$) -- (A31);
\draw [vec] (A21) -- ($0.3*(A21)+0.7*(A32)$); \draw ($0.3*(A21)+0.7*(A32)$) -- (A32);
\draw [vec] (A22) -- ($0.3*(A22)+0.7*(A31)$); \draw ($0.3*(A22)+0.7*(A31)$) -- (A31);
\draw [vec] (A22) -- ($0.3*(A22)+0.7*(A33)$); \draw ($0.3*(A22)+0.7*(A33)$) -- (A33);
\draw [vec] (A23) -- ($0.3*(A23)+0.7*(A32)$); \draw ($0.3*(A23)+0.7*(A32)$) -- (A32);
\draw [vec] (A23) -- ($0.3*(A23)+0.7*(A33)$); \draw ($0.3*(A23)+0.7*(A33)$) -- (A33);
\draw [vec] (A31) -- ($0.3*(A31)+0.7*(A4)$); \draw ($0.3*(A31)+0.7*(A4)$) -- (A4);
\draw [vec] (A32) -- ($0.3*(A32)+0.7*(A4)$); \draw ($0.3*(A32)+0.7*(A4)$) -- (A4);
\draw [vec] (A33) -- ($0.3*(A33)+0.7*(A4)$); \draw ($0.3*(A33)+0.7*(A4)$) -- (A4);
\end{tikzpicture}
\caption{\(R_P\) is white, \(X_0 \times X_0\) is red, \(X_1 \times X_1\) is yellow, \(X_2 \times X_2\) is blue, \(X_0^2 \cup X_1^2\) is orange, \(X_0^2 \cup X_2^2\) is green, \(X_1^2 \cup X_2^2\) is violet, and \(\varnothing\) is black.}
\label{fig2.3}
\end{figure}

Let \(X\) be a set, let \(R\) be a nonempty binary relation on \(X\). Write \(X_1\) and \(X_2\) for the \emph{domain} and, respectively, for the \emph{range} of the relation \(R\), i.e., a point \(x\) belongs to \(X_1\) (\(X_2\)) if and only if there is \(x_2 \in X\) (\(x_1 \in X\)) such that \(\<x, x_2> \in R\) (\(\<x_1, x> \in R\)).

\begin{lemma}\label{l4.13}
Let \(R\) be a binary relation with a domain \(X_1\) and a range \(X_2\). The equality \(R \circ R = \varnothing\) holds if and only if \(X_1 \cap X_2 = \varnothing\).
\end{lemma}

\begin{proof}
It follows directly from the definition of the composition \(\circ\) of binary relations.
\end{proof}

\begin{proposition}\label{p4.14}
Let \(Y\) be a set with \(|Y| \geqslant 3\) and let \(Q = \{\Delta_{Y}, \nabla_{Y}\}\), where \(\Delta_{Y}\) is the diagonal of \(Y\) and \(\nabla_{Y} = Y^{2} \setminus \Delta_{Y}\). Then \(Q\) is a partition of \(Y^{2}\) and the subsemigroup \((\mc{S}_{Q}, \circ)\) of \(\mc{B}_{Y}\) has no \(d\)-transitive representations \(\mc{S}_{Q} \to \mc{B}_X\) for any set \(X\).
\end{proposition}

\begin{proof}
It is clear that \(Q\) is a partition of \(Y^{2}\). Suppose there is a \(d\)-transitive monomorphism \(\Phi \colon \mc{S}_{Q} \to \mc{B}_X\) for a suitable set \(X\). Direct calculations show that
\[
\mc{S}_{Q} = \{\Delta_{Y}, \nabla_{Y}, Y^{2}\} \quad \text{and} \quad \nabla_{Y} \circ \nabla_{Y} = Y^{2}
\]
hold, and \(\Delta_{Y}\) is the identity of \(\mc{S}_{Q}\), and \(Y^{2}\) is the zero of \(\mc{S}_{Q}\). Let \(A\) be a set of generators of \(\mc{S}_{Q}\) for which \(\Phi(A)\), \(\Phi(A) = \{\Phi(a) \colon a \in A\}\), is a partition of \(X^{2}\). Since \(Y^{2}\) is a zero of \(\mc{S}_{Q}\) and \(\Phi\) is \(d\)-transitive, the equality \(\Phi(Y^{2}) = \varnothing\) holds. It implies \(Y^{2} \notin A\). There are exactly two sets, \(\{\Delta_{Y}, \nabla_{Y}\}\) and \(\{\Delta_{Y}, \nabla_{Y}, Y^{2}\}\), of generators of \((\mc{S}_{Q}, \circ)\). Consequently, we have \(A = \{\Delta_{Y}, \nabla_{Y}\}\).

Let \(X_1\) and \(X_2\) be the domain and, respectively, the range of the relation \(\Phi(\nabla_{Y})\). We claim that the equality
\begin{equation}\label{p4.14:e1}
\Phi(\Delta_{Y}) = X_1^{2} \cup X_2^{2}
\end{equation}
holds.

Let us prove the last equality. Lemma~\ref{l4.13} implies
\begin{equation}\label{p4.14:e2}
X_1 \cap X_2 = \varnothing
\end{equation}
and, moreover, from the definition of \(X_1\) and \(X_2\) it follows that
\begin{equation}\label{p4.14:e3}
\Phi(\nabla_{Y}) \subseteq X_1 \times X_2.
\end{equation}
Since \(\Phi\) is a monomorphism, we have
\begin{equation}\label{p4.14:e4}
\Phi(\nabla_{Y}) = \Phi(\Delta_{Y}) \circ \Phi(\nabla_{Y}) = \Phi(\nabla_{Y}) \circ \Phi(\Delta_{Y}).
\end{equation}
Let \(z\) be an arbitrary point of \(X\) and let \(x_1\) be an arbitrary point of \(X_1\). Suppose that \(z \notin X_1 \cup X_2\), then \eqref{p4.14:e3} implies \(\<z, x_1> \in \Phi(\Delta_{Y})\) because \(\{\Phi(\Delta_{Y}), \Phi(\nabla_{Y})\}\) is a partition of \(X^{2}\). Since \(X_2\) is the range of \(\Phi(\nabla_{Y})\), there is \(x_2 \in X_2\) such that \(\<x_1, x_2> \in \Phi(\nabla_{Y})\). Now, using~\eqref{p4.14:e4}, we obtain \(\<z, x_2> \in \Phi(\nabla_{Y})\). Consequently, \(z \in X_1\), that contradicts \(z \notin X_1 \cup X_2\). Thus, the equality \(X = X_1 \cup X_2\) holds. The last equality and \eqref{p4.14:e2} imply the double inclusion
\begin{equation}\label{p4.14:e5}
X^2 \supseteq \Phi(\Delta_{Y}) \supseteq X_1^{2} \cup X_2^{2}.
\end{equation}
If the set \(\Phi(\Delta_{Y}) \setminus (X_1^{2} \cup X_2^{2})\) is nonempty, then using~\eqref{p4.14:e5} and the equality \(X^{2} = \Phi(\Delta_{Y}) \cup \Phi(\nabla_{Y})\) we can find \(t_1 \in X_1\) and \(t_2 \in X_2\) such that \(\<t_1, t_2> \in \Phi(\Delta_{Y})\) or \(\<t_2, t_1> \in \Phi(\Delta_{Y})\). Without loss of generality we may suppose
\begin{equation}\label{p4.14:e6}
\<t_2, t_1> \in \Phi(\Delta_{Y}).
\end{equation}
Since \(X_1\) is the domain of \(\Phi(\nabla_{Y})\), from \(t_2 \in X_2\) it follows that there is \(x \in X_1\) such that
\begin{equation}\label{p4.14:e7}
\<x, t_2> \in \Phi(\nabla_{Y}).
\end{equation}
Now \eqref{p4.14:e4}, \eqref{p4.14:e6} and \eqref{p4.14:e7} give us \(\<x, t_1> \in \Phi(\nabla_{Y})\). Consequently, \(t_1 \in X_2\) holds, contrary to \(t_1 \in X_1\). Equality~\eqref{p4.14:e1} follows.

Equality~\eqref{p4.14:e1}, \(X = X_1 \cup X_2\), and inclusion~\eqref{p4.14:e3} imply
\[
X^{2} \setminus (\Phi(\Delta_{Y}) \cup \Phi(\nabla_{Y})) \supseteq X^{2} \setminus (X_1^{2} \cup X_2^{2} \cup (X_1 \times X_2)) = X_2 \times X_1 \neq \varnothing,
\]
i.e., \(\{\Phi(\Delta_{Y}), \Phi(\nabla_{Y})\}\) is not a partition of \(X^{2}\).

Thus, contrary to our supposition, \(\Phi\) is not a \(d\)-transitive monomorphism.
\end{proof}

\begin{remark}\label{r4.15}
If \(|Y| = 2\) and \(Q = \{\nabla_{Y}, \Delta_{Y}\}\), then the equalities
\[
\Delta_{Y} \circ \nabla_{Y} = \nabla_{Y}  = \nabla_{Y} \circ \Delta_{Y} \quad \text{and} \quad \nabla_{Y} \circ \nabla_{Y} = \Delta_{Y} = \Delta_{Y} \circ \Delta_{Y}
\]
hold. Consequently, \((\mc{S}_{Q}, \circ)\) is a group of order \(2\) for which the identity mapping \(\Id \colon \mc{S}_{Q} \to \mc{B}_Y\) is a \(d\)-transitive monomorphism.
\end{remark}

The last remark shows, in particular, that the constant \(3\) cannot be replaced by any smaller integer in Proposition~\ref{p4.14}.

\medskip
\textbf{Acknowledgments.} This research was partially supported by State Fund for Fundamental Research of Ukraine, Project F75/28173. The author thanks Prof. Ruslan Shanin, Odessa I.~I.~Mechnikov National University, for the useful discussions of the presented results.


\end{document}